\documentclass[11pt]{article}
\usepackage{amsfonts,amscd,amssymb, amsmath}

\allowdisplaybreaks[4]
\newtheorem{definition}{Definition}[section]
\newtheorem{lemma}[definition]{Lemma}
\newtheorem{proposition}[definition]{Proposition}

\newtheorem{remark}[definition]{Remark}

\newtheorem{theorem}[definition]{Theorem}
\newtheorem{example}[definition]{Example}

\def\rawo\lonra{\longrightarrow}

\def\ot{\otimes}

\newcommand{\selabel}[1]{\label{se:#1}}

\allowdisplaybreaks[4]

\newenvironment{proof}{{\it Proof.}}{\hfill $ \square $ \vskip 4mm}

\begin{document}
\title{Twisting operators, twisted tensor products and smash products\\ for Hom-associative algebras}
\author{Abdenacer Makhlouf\\
Universit\'{e} de Haute Alsace, \\
Laboratoire de Math\'{e}matiques, Informatique et Applications, \\
4, Rue des Fr\`{e}res Lumi\`{e}re, F-68093 Mulhouse, France\\
e-mail: Abdenacer.Makhlouf@uha.fr
\and Florin Panaite\thanks {Work supported by a grant of the Romanian National 
Authority for Scientific Research, CNCS-UEFISCDI, 
project number PN-II-ID-PCE-2011-3-0635,  
contract nr. 253/5.10.2011.}\\
Institute of Mathematics of the
Romanian Academy\\
PO-Box 1-764, RO-014700 Bucharest, Romania\\
 e-mail: Florin.Panaite@imar.ro}
\date{}
\maketitle

\begin{abstract}
The purpose of this paper is to provide new constructions of Hom-associative algebras using Hom-analogues of 
certain operators called twistors and pseudotwistors, by deforming a given Hom-associative multiplication into a 
new Hom-associative multiplication. As examples, we introduce Hom-analogues of the twisted tensor product and  
smash product. Furthermore we show that the construction by the ''twisting principle'' introduced by Yau and the twisting 
of associative algebras using pseudotwistors admit a common generalization.
\end{abstract}
%%%%%%%%%%%%%%%%%%%%%%%%%%%%%%
\section*{Introduction}
%%%%%%%%%%%%%%%%%%%%%%%%%%%%%%
${\;\;\;}$
The motivation to introduce  Hom-type algebras comes for examples related to $q$-deformations of  
Witt and Virasoro 
algebras,
 which play an important r\^{o}le in Physics, mainly in conformal field theory. A $q$-deformation of an algebra 
of vector fields is  obtained when  the  derivation is replaced by a 
$\sigma$-derivation. It was observed in the pioneering works   
\cite{Alvarez,AizawaSaito,ChaiElinPop,ChaiKuLukPopPresn,ChaiIsKuLuk,ChaiPopPres,
CurtrZachos1,DaskaloyannisGendefVir,Hu, Kassel1,LiuKeQin} that  $q$-deformations of Witt and Virasoro 
algebras are no longer  Lie algebras, but satisfy a twisted Jacobi condition.
Motivated by these
examples and their generalization, Hartwig, Larsson and Silvestrov in \cite{HLS,LS1,LS2,LS3} 
introduced the notion of  Hom-Lie algebra as a deformation of Lie algebras in which the Jacobi identity 
is twisted by a homomorphism. The  associative-type objects corresponding to Hom-Lie algebras, called 
 Hom-associative algebras, have been introduced in \cite{ms1}. Usual functors between the categories of 
Lie algebras and associative algebras have been extended to the Hom-setting. It was shown  
that a commutator of a Hom-associative algebra gives rise to a Hom-Lie algebra; the construction of the free 
Hom-associative algebra and the  enveloping algebra of a Hom-Lie algebra have been provided in 
\cite{Yau:EnvLieAlg}. Since then, Hom-analogues of various classical structures and results  have 
been introduced and discussed by many authors. For instance, representation theory,   cohomology and 
deformation theory  for Hom-associative algebras and Hom-Lie algebras have been developed in 
\cite{AEM,ms2,Sheng}. For further results see \cite{ AmmarMakhlouf2009,BenayadiMakhlouf,LiuChenMa,Mak:Almeria} and  \cite{fgs, Gohr} for other properties of Hom-associative algebras.
All these generalizations coincide with the usual definitions when the structure map equals the identity.

The dual concept of Hom-associative algebras, called Hom-coassociative coalgebras, as well as 
Hom-bialgebras and Hom-Hopf algebras, have been introduced in \cite{ms3,ms4} and also studied  in 
\cite{stef,yau2}.  As expected, the enveloping Hom-associative algebra of a Hom-Lie algebra is 
naturally a Hom-bialgebra. A twisted version of module algebras called module Hom-algebras has 
been studied in \cite{yau1}, 
where $q$-deformations of the $\mathfrak{sl}(2)$-action on the affine plane were provided. 
Objects admitting coactions by Hom-bialgebras have been studied first in \cite{yau2}. 
 In \cite{yauhomyb1,yauhomyb2,yauhomyb3}, various generalizations of Yang-Baxter equations and 
related algebraic structures have been studied. D. Yau provided solutions of HYBE, a twisted 
version of the Yang-Baxter equation called the Hom-Yang-Baxter equation, from Hom-Lie algebras, 
quantum enveloping algebra of $\mathfrak{sl}(2)$, the Jones-Conway polynomial, 
Drinfeld's (co)quasitriangular bialgebras and Yetter-Drinfeld modules (over bialgebras). 
Yetter-Drinfeld modules over Hom-bialgebras and their category have been studied in \cite{PanaiteMakhlouf}. 
For further results about generalizations of quantum groups and related structures see \cite{BEM,homquantum1,homquantum2,homquantum3}.
In \cite{Elhamdadi-Makhlouf}, Hom-quasi-bialgebras have been introduced and concepts like gauge 
transformation and Drinfeld twist  generalized.
Moreover,  an example of a twisted quantum double was provided.  

One of the main tools to construct
examples of  Hom-type algebras is the ''twisting principle'' (called sometimes ''composition method''). 
It was introduced by D. Yau for Hom-associative algebras and since then 
extended to various Hom-type algebras. It allows to construct a Hom-type algebra starting from a 
classical-type algebra and an algebra homomorphism.

The twisted tensor product $A\ot _RB$ of two associative algebras $A$ and $B$ is a certain algebra 
structure on the vector space $A\ot B$, defined in terms of a so-called twisting map $R:B\ot A\rightarrow 
A\ot B$, having the property that it coincides with the usual tensor product algebra $A\ot B$ if $R$ is the usual flip 
map. This construction was introduced in \cite{Cap,VanDaele} and it may be regarded as a 
representative for the Cartesian product of noncommutative spaces. An important example of a twisted 
tensor product of associative algebras is a smash product $A\# H$, where $H$ is a bialgebra and 
$A$ is a left $H$-module algebra. Motivated by the desire to express the multiplication of 
$A\ot _RB$ as a ''deformation'' of the multiplication 
of $A\ot B$, in \cite{lpvo} was introduced the concept of {\em pseudotwistor} (with a particular case 
called {\em twistor}) for an associative algebra $D$, with multiplication $\mu :D\ot D\rightarrow D$, 
as a linear map $T:D\ot D\rightarrow D\ot D$ 
satisfying some axioms that imply that the new multiplication $\mu \circ T$ on $D$ is also associative. It turns 
out that many other ''deformed multiplications'' that appear in the literature (such as twisted bialgebras and 
Fedosov products) are afforded by such pseudotwistors.

The aim of this paper is to introduce Hom-analogues of twistors, pseudotwistors and twisted tensor products and 
to use them to obtain new Hom-associative algebras starting with one or more given Hom-associative algebras. 

The paper is organized as follows. In Section 1 we review the main definitions and results about  
twisting associative algebras by means of twistors and pseudotwistors and the basics on  Hom-associative algebras, 
Hom-bialgebras and related structures. In Section 2 we introduce the concepts of Hom-twistor, 
Hom-pseudotwistor, Hom-twisting map and Hom-twisted tensor product of Hom-associative algebras; we prove 
that these concepts are compatible with the twisting principle and that the Hom-twisted tensor product can be 
iterated.  Section 3 deals with smash products in the Hom-setting. Given a Hom-bialgebra $H$ and a 
left (respectively right)  $H$-module Hom-algebra $A$ (respectively $C$) such that all structure maps 
$\alpha _H$, $\alpha _A$ (respectively $\alpha _C$) are bijective, we define in a natural way a Hom-twisting 
map $R$ between $A$ and $H$ (respectively between $H$ and $C$) and a Hom-associative algebra 
$A\# H:=A\otimes_R H$ (respectively $H\#C:=H\ot _RC$), called the left (respectively right) 
Hom-smash product. Given both $A$ and $C$ as above, we define also the so-called two-sided Hom-smash product 
$A\# H\# C$. 

In the last section we show that  Yau's procedure of obtaining a Hom-associative algebra from an associative 
algebra via the twisting principle and the procedure of obtaining a new associative algebra from a given 
associative algebra via a pseudotwistor admit a common generalization, by means of a new concept called $\alpha$-pseudotwistor, where $\alpha $ is an algebra endomorphism of an associative algebra.
%%%%%%%%%%%%%%%%%%%%%%%%%%%%%%%
\section{Preliminaries}\selabel{1}
%%%%%%%%%%%%%%%%%%%%%%%%%%%%%%%
${\;\;\;}$
We work over a base field $k$. All algebras, linear spaces
etc. will be over $k$; unadorned $\ot $ means $\ot_k$. For a comultiplication 
$\Delta :C\rightarrow C\ot C$ on a vector space $C$ we use a 
Sweedler-type notation $\Delta (c)=c_1\ot c_2$, for $c\in C$. Unless 
otherwise specified, the (co)algebras ((co)associative or not) that will appear 
in what follows are {\em not} supposed to be (co)unital, and a multiplication 
$\mu :V\ot V\rightarrow V$ on a linear space $V$ is denoted by juxtaposition: 
$\mu (v\ot v')=vv'$. 

We recall some concepts and results, fixing the terminology 
to be used throughout the paper. 
\begin{definition}(\cite{Cap}, \cite{VanDaele}) Let $(A, \mu _A)$, $(B, \mu _B)$ 
be two associative algebras. A {\em twisting map} between $A$ and $B$ is a linear map $R:B\ot A 
\rightarrow A\ot B$ satisfying the conditions: 
\begin{eqnarray}
&&R\circ (id_B\ot \mu _A)=(\mu _A\ot id_B)\circ (id_A\ot R)\circ (R\ot id_A),
\label{twmap1}\\
&&R\circ (\mu _B\ot id_A)=(id_A\ot \mu _B)\circ (R\ot id_B)\circ (id_B\ot R).
\label{twmap2}
\end{eqnarray}
If this is the case, the map $\mu _R=(\mu _A\ot \mu _B)\circ (id_A\ot R\ot id_B)$ 
is an associative product on $A\ot B$; the associative algebra $(A\ot B, \mu _R)$ 
is denoted by $A\ot _RB$ and called the {\em twisted tensor product} of $A$ and $B$ 
afforded by $R$. 
\end{definition}

If we use a Sweedler-type notation $R(b\ot a)=a_R\ot b_R=a_r\ot b_r$, for $a\in A$, $b\in B$, then
(\ref{twmap1}) and (\ref{twmap2}) may be rewritten as:
\begin{eqnarray}
&&(aa')_R\ot b_R=a_Ra'_r\ot (b_R)_r, \label{tw4} \\
&&a_R\ot (bb')_R=(a_R)_r\ot b_rb'_R, \label{tw5}
\end{eqnarray}
and the multiplication of $A\ot _RB$ may be written as $(a\ot b)(a'\ot b')=aa'_R\ot b_Rb'$. 

\begin{example} We construct  a twisted  tensor product of  $k^2\otimes k^2$.  
The multiplication of $k^2$ with respect to $\{e_1,e_2\} $ is defined as $e_ie_j=\delta_{i,j}e_i$ for 
$i,j=1,2$, where $\delta$ is the Kronecker symbol. 
We provide a one-parameter family of twisting maps ($\lambda $ is a parameter in $k$):
\begin{eqnarray*}
&& R(e_1\otimes e_1)=\lambda e_1\otimes e_1+\lambda e_1\otimes e_2+\lambda e_2\otimes e_1+(\lambda -1) e_2\otimes e_2,\\
&& R(e_1\otimes e_2)=(1-\lambda) e_1\otimes e_1-\lambda e_1\otimes e_2+(1-\lambda) e_2\otimes e_1+(1-\lambda ) e_2\otimes e_2,\\
&& R(e_2\otimes e_1)=(1-\lambda) e_1\otimes e_1+(1-\lambda) e_1\otimes e_2-\lambda e_2\otimes e_1+(1-\lambda ) e_2\otimes e_2,\\
&&  R(e_2\otimes e_2)=(\lambda-1) e_1\otimes e_1+\lambda e_1\otimes e_2+\lambda e_2\otimes e_1+\lambda  e_2\otimes e_2.
\end{eqnarray*}
Therefore, we obtain the following new multiplication on  $k^2\otimes k^2$:
%\begin{table}[htdp]
%\caption{new multiplication on  $k^2\otimes k^2$}
\begin{center}
\begin{tabular}{c|cccc}
\ & $e_1\otimes e_1$&$ e_1\otimes e_2 $& $e_2\otimes e_1$ &$ e_2\otimes e_2$\\
\hline 
$e_1\otimes e_1 $&$\lambda  e_1\otimes e_1$&$ \lambda e_1\otimes e_2 $ &$(1-\lambda) e_1\otimes e_1 $&$-\lambda  e_1\otimes e_2$ \\
$e_1\otimes e_2 $&$(1-\lambda)  e_1\otimes e_1$&$(1- \lambda) e_1\otimes e_2 $ &$(\lambda-1) e_1\otimes e_1 $&$\lambda  e_1\otimes e_2$ \\
$e_2\otimes e_1 $&$\lambda  e_2\otimes e_1$&$( \lambda-1) e_2\otimes e_2 $ &$(1-\lambda) e_2\otimes e_1 $&$(1-\lambda)  e_2\otimes e_2$ \\
$e_2\otimes e_2 $&$-\lambda  e_2\otimes e_1$&$(1- \lambda) e_2\otimes e_2 $ &$\lambda e_2\otimes e_1 $&$\lambda  e_2\otimes e_2$
\end{tabular}
\end{center}
\label{default}
%\end{table}%
\end{example}

The following two concepts are versions for nonunital algebras of the ones introduced in \cite{lpvo}: 

\begin{definition} 
Let $(D, \mu )$ be an associative algebra and $T:D\otimes
D\rightarrow D\otimes D$ a linear map. Assume that there exist 
two linear maps $\tilde{T}_1,
\tilde{T}_2:D\otimes D\otimes D \rightarrow D\otimes D\otimes D$ such that the following 
conditions are satisfied:
\begin{eqnarray}
&&T\circ (id_D\otimes \mu )=
(id_D\otimes \mu )\circ \tilde{T}_1\circ (T\otimes id_D),
\label{pstw1} \\
&&T\circ (\mu \otimes id_D)=
(\mu \otimes id_D)\circ \tilde{T}_2\circ (id_D\otimes T),
\label{pstw2} \\
&&\tilde{T}_1\circ (T\otimes id_D)\circ (id_D\otimes T)=
\tilde{T}_2\circ (id_D\otimes T)\circ (T\otimes id_D). \label{pstw3}
\end{eqnarray}
Then $D^T:=(D, \mu \circ T)$ is also an associative algebra. The map $T$ is called a
{\em pseudotwistor} and the two maps $\tilde{T}_1$,
$\tilde{T}_2$ are called the {\em companions} of $T$.
\end{definition}
\begin{definition} 
Let $(D, \mu )$ be an associative algebra and $T:D\otimes
D\rightarrow D\otimes D$ a linear map, with Sweedler-type notation $T(d\ot d')=d^T\ot d'_T$, 
for $d, d'\in D$,  satisfying the following conditions:
\begin{eqnarray}
&&T\circ (id_D\otimes \mu )=
(id_D\otimes \mu )\circ T_{13}\circ T_{12},
\label{twistor1} \\
&&T\circ (\mu \otimes id_D)=
(\mu \otimes id_D)\circ T_{13}\circ T_{23},
\label{twistor2} \\
&&T_{12}\circ T_{23}=T_{23}\circ T_{12},  \label{twistor3}
\end{eqnarray}
where we used a standard notation for the operators $T_{ij}$, namely $T_{12}=T\ot id_D$, 
$T_{23}=id_D\ot T$ and $T_{13}(d\ot d'\ot d'')=d^T\ot d'\ot d''_T$. 
Then $D^T:=(D, \mu \circ T)$ is also an associative algebra, and the map $T$ is called a
{\em twistor} for $D$. 
\end{definition}

Obviously, any twistor $T$ is a pseudotwistor with companions 
$\tilde{T}_1=\tilde{T}_2=T_{13}$. 

If $A\ot _RB$ is a twisted tensor product of associative algebras, the map 
$T:(A\ot B)\ot (A\ot B)\rightarrow (A\ot B)\ot (A\ot B)$, $T((a\ot b)\ot (a'\ot b'))=
(a\ot b_R)\ot (a'_R\ot b')$, is a twistor for the ordinary tensor product algebra $A\ot B$ 
and $A\ot _RB=(A\ot B)^T$ as associative algebras, cf. \cite{lpvo}. 

We recall now several things about Hom-structures. Since various authors use different terminology, 
some caution is necessary. In what follows, we use terminology as in our previous paper 
\cite{PanaiteMakhlouf}.
\begin{definition}
(i) A {\em Hom-associative algebra} is a triple $(A, \mu , \alpha )$, in which $A$ is a linear space, 
$\alpha :A\rightarrow A$  and $\mu :A\ot A\rightarrow A$ are linear maps,  
with notation $\mu (a\ot a')=aa'$, satisfying the following conditions, for all $a, a', a''\in A$:
\begin{eqnarray*}
&&\alpha (aa')=\alpha (a)\alpha (a'), \;\;\;\;\;(multiplicativity)\\
&&\alpha (a)(a'a'')=(aa')\alpha (a''). \;\;\;\;\;(Hom-associativity)
\end{eqnarray*}
We call $\alpha $ the {\em structure map} of $A$. 

A morphism $f:(A, \mu _A , \alpha _A)\rightarrow (B, \mu _B , \alpha _B)$ of Hom-associative algebras 
is a linear map $f:A\rightarrow B$ such that $\alpha _B\circ f=f\circ \alpha _A$ and 
$f\circ \mu_A=\mu _B\circ (f\ot f)$. \\
(ii) A {\em Hom-coassociative coalgebra} is a triple $(C, \Delta, \alpha )$, in which $C$ is a linear 
space, $\alpha :C\rightarrow C$ and $\Delta :C\rightarrow C\ot C$ are linear maps, 
satisfying the following conditions:
\begin{eqnarray*}
&&(\alpha \ot \alpha )\circ \Delta =
\Delta \circ \alpha , \;\;\;\;\;(comultiplicativity)\\  
&&(\Delta \ot \alpha )\circ \Delta =
(\alpha \ot \Delta )\circ \Delta . \;\;\;\;\;(Hom-coassociativity)
\end{eqnarray*}

A morphism $g:(C, \Delta _C , \alpha _C)\rightarrow (D, \Delta _D , \alpha _D)$ of Hom-coassociative 
coalgebras  
is a linear map $g:C\rightarrow D$ such that $\alpha _D\circ g=g\circ \alpha _C$ and 
$(g\ot g)\circ \Delta _C=\Delta _D\circ g$.
\end{definition}
\begin{remark}
Assume that $(A, \mu _A , \alpha _A)$ and $(B, \mu _B, \alpha _B)$ are two Hom-associative algebras; then 
$(A\ot B, \mu _{A\ot B}, \alpha _A\ot \alpha _B)$ is a Hom-associative algebra (called the tensor 
product of $A$ and $B$), where $\mu _{A\ot B}$ is the usual multiplication: 
$(a\ot b)(a'\ot b')=aa'\ot bb'$. 
\end{remark}
\begin{definition} (\cite{yau1}, \cite{homquantum3}) 
(i) Let $(A, \mu _A , \alpha _A)$ be a Hom-associative algebra, $M$ a linear space and $\alpha _M:M
\rightarrow M$ a linear map. A {\em left $A$-module} structure on $(M, \alpha _M)$ consists of a linear map 
$A\ot M\rightarrow M$, $a\ot m\mapsto a\cdot m$, satisfying the conditions:
\begin{eqnarray}
&&\alpha _M(a\cdot m)=\alpha _A(a)\cdot \alpha _M(m), \label{hommod1}\\
&&\alpha _A(a)\cdot (a'\cdot m)=(aa')\cdot \alpha _M(m), \label{hommod2}
\end{eqnarray} 
for all $a, a'\in A$ and $m\in M$. If $(M, \alpha _M)$ and $(N, \alpha _N)$ are left $A$-modules (both 
$A$-actions denoted by $\cdot$),  
a morphism of left $A$-modules $f:M\rightarrow N$ is a linear map satisfying the conditions 
$\alpha _N\circ f=f\circ \alpha _M$  and $f(a\cdot m)=a\cdot f(m)$, for all $a\in A$ and $m\in M$. \\
(ii) Let $(C, \Delta _C , \alpha _C)$ be a Hom-coassociative coalgebra, $M$ a linear space and $\alpha _M:M
\rightarrow M$ a linear map. A {\em left $C$-comodule} structure on $(M, \alpha _M)$ consists of a linear map 
$\lambda :M\rightarrow C\ot M$ satisfying the following conditions:
\begin{eqnarray}
&&(\alpha _C\ot \alpha _M)\circ \lambda =\lambda \circ \alpha _M,  
\label{leftcom1}\\
&&(\Delta _C\ot \alpha _M)\circ \lambda =(\alpha _C\ot \lambda )\circ \lambda . 
\label{leftcom2}
\end{eqnarray} 
If $(M, \alpha _M)$ and $(N, \alpha _N)$ are left $C$-comodules, with structures 
$\lambda _M:M\rightarrow C\ot M$ and $\lambda _N:N\rightarrow C\ot N$, 
a morphism of left $C$-comodules $g:M\rightarrow N$ is a linear map satisfying the conditions 
$\alpha _N\circ g=g\circ \alpha _M$  and $(id_C\ot g)\circ \lambda _M=\lambda _N\circ g$. 
\end{definition}

\begin{definition} (\cite{ms3}, \cite{ms4})
A {\em Hom-bialgebra} is a quadruple $(H, \mu , \Delta, \alpha )$, in which $(H, \mu , \alpha )$ is 
a Hom-associative algebra, $(H, \Delta , \alpha )$ is a Hom-coassociative coalgebra  
and moreover $\Delta $ is a morphism of Hom-associative algebras.  
\end{definition}

In other words, a Hom-bialgebra is a Hom-associative algebra $(H, \mu , \alpha )$ endowed with a 
linear map $\Delta :H\rightarrow H\ot H$, with notation $\Delta (h)=h_1\ot h_2$, such that the 
following conditions are satisfied, for all $h, h'\in H$: 
\begin{eqnarray}
&&\Delta (h_1)\ot \alpha (h_2)=\alpha (h_1)\ot \Delta (h_2),  \label{hombia1}\\
&&\Delta (hh')=h_1h'_1\ot h_2h'_2, \label{hombia2}\\
&&\Delta (\alpha (h))=\alpha (h_1)\ot \alpha (h_2). \label{hombia3}
\end{eqnarray}
\begin{proposition} (\cite{ms4}, \cite{yau3}) \label{yautwisting}
(i) Let $(A, \mu )$ be an associative algebra and $\alpha :A\rightarrow A$ an algebra endomorphism. Define 
a new multiplication $\mu _{\alpha }:=\alpha \circ \mu :A\ot A\rightarrow A$. Then 
$(A, \mu _{\alpha }, \alpha )$ is a Hom-associative algebra, denoted by $A_{\alpha }$. \\
(ii) Let $(C, \Delta )$ be a coassociative coalgebra and $\alpha :C\rightarrow C$ a 
coalgebra endomorphism. Define 
a new comultiplication $\Delta _{\alpha }:=\Delta \circ \alpha :C\rightarrow C\ot C$. Then 
$(C, \Delta _{\alpha }, \alpha )$ is a Hom-coassociative coalgebra, denoted by $C_{\alpha }$. \\
(iii) Let $(H, \mu , \Delta )$ be a bialgebra and $\alpha :H\rightarrow H$ a bialgebra endomorphism. 
If we define $\mu _{\alpha }$ and $\Delta _{\alpha }$ as in (i) and (ii), then 
$H_{\alpha }=(H, \mu _{\alpha }, \Delta _{\alpha }, \alpha )$ is a Hom-bialgebra. 
\end{proposition}
\begin{proposition} (\cite{homquantum3})
Let $(H, \mu _H, \Delta _H, \alpha _H)$ be a Hom-bialgebra and $(M, \alpha _M)$ and 
$(N, \alpha _N)$ two left $H$-modules. Then $(M\ot N, \alpha _M\ot \alpha _N)$ is also a 
left $H$-module, with $H$-action defined by $H\ot (M\ot N)\rightarrow M\ot N$, $h\ot (m\ot n)\mapsto 
h\cdot (m\ot n):=h_1\cdot m\ot h_2\cdot n$. 
\end{proposition}
\begin{definition} (\cite{yau2})
Let $(H, \mu _H, \Delta _H, \alpha _H)$ be a Hom-bialgebra. A {\em left $H$-comodule Hom-algebra} 
is a Hom-associative algebra $(D, \mu _D, \alpha _D)$ endowed with a left $H$-comodule 
structure $\lambda _D:D\rightarrow H\ot D$ such that $\lambda _D$ is a morphism of 
Hom-associative algebras.
\end{definition}
\begin{definition} (\cite{yau1}) 
Let $(H, \mu _H, \Delta _H, \alpha _H)$ be a Hom-bialgebra. A Hom-associative algebra 
$(A, \mu _A, \alpha _A)$ is called a {\em left $H$-module Hom-algebra} if $(A, \alpha _A)$ is a 
left $H$-module, 
with action denoted by $H\ot A\rightarrow A$, $h\ot a\mapsto h\cdot a$, such that the following 
condition is satisfied: 
\begin{eqnarray}
&&\alpha _H^2(h)\cdot (aa')=(h_1\cdot a)(h_2\cdot a'), \;\;\;\forall \;h\in H, \;a, a'\in A. \label{modalgcompat}
\end{eqnarray} 
\end{definition}

One may wonder why it was chosen $\alpha _H^2$ in the above formula (and not, for instance, 
$\alpha _H$). The answer is provided by the following result:
\begin{proposition} (\cite{yau1}) \label{deformmodalg}
Let $(H, \mu _H, \Delta _H)$ be a bialgebra and $(A, \mu _A)$ a left $H$-module algebra in the usual sense, 
with action denoted by $H\ot A\rightarrow A$, $h\ot a\mapsto h\cdot a$. Let $\alpha _H:H\rightarrow H$ 
be a bialgebra endomorphism and $\alpha _A:A\rightarrow A$ an algebra endomorphism, such that 
$\alpha _A(h\cdot a)=\alpha _H(h)\cdot \alpha _A(a)$, for all $h\in H$ and $a\in A$. If we 
consider the Hom-bialgebra 
$H_{\alpha _H}=(H, \alpha _H\circ \mu _H, \Delta _H\circ \alpha _H, \alpha _H)$ and 
the Hom-associative algebra $A_{\alpha _A}=(A, \alpha _A\circ \mu _A, \alpha _A)$, 
then $A_{\alpha _A}$ is a left $H_{\alpha _H}$-module Hom-algebra in the above sense, with action 
$H_{\alpha _H}\ot A_{\alpha _A}\rightarrow A_{\alpha _A}$, $h\ot a\mapsto h\triangleright a:=
\alpha _A(h\cdot a)=\alpha _H(h)\cdot \alpha _A(a)$. 
\end{proposition}
%%%%%%%%%%%%%%%%%%%%%%%%%%%%%%%%%%%%%%
\section{Hom-pseudotwistors and Hom-twisted tensor products}
%%%%%%%%%%%%%%%%%%%%%%%%%%%%%
\setcounter{equation}{0}
%%%%%%%%%%%%%%%%%%%%%%%%%%%%
${\;\;\;}$We begin by introducing the Hom-analogues of twistors and pseudotwistors. 
\begin{proposition} \label{hompseudotwistor}
Let $(D, \mu , \alpha )$ be a Hom-associative algebra and $T:D\otimes
D\rightarrow D\otimes D$ a linear map. Assume that there exist 
two linear maps $\tilde{T}_1,
\tilde{T}_2:D\otimes D\otimes D \rightarrow D\otimes D\otimes D$ such that the following 
relations hold:
\begin{eqnarray}
&&(\alpha \ot \alpha )\circ T=T\circ (\alpha \ot \alpha ), \label{hommultT}\\
&&T\circ (\alpha \otimes \mu )=
(\alpha \otimes \mu )\circ \tilde{T}_1\circ (T\otimes id_D),
\label{hompstw1} \\
&&T\circ (\mu \otimes \alpha )=
(\mu \otimes \alpha )\circ \tilde{T}_2\circ (id_D\otimes T),
\label{hompstw2} \\
&&\tilde{T}_1\circ (T\otimes id_D)\circ (id_D\otimes T)=
\tilde{T}_2\circ (id_D\otimes T)\circ (T\otimes id_D). \label{hompstw3}
\end{eqnarray}
Then $D^T:=(D, \mu \circ T, \alpha )$ is also a Hom-associative algebra. The map $T$ is called a
{\em Hom-pseudotwistor} and the two maps $\tilde{T}_1$,
$\tilde{T}_2$ are called the {\em companions} of $T$.
\end{proposition}
\begin{proof}
We record first the obvious relations
\begin{eqnarray}
&&(\mu \circ T)\ot \alpha =(\mu \ot \alpha )\circ (T\ot id_D), \label{ajut1} \\
&&\alpha \ot (\mu \circ T)=(\alpha \ot \mu )\circ (id_D\ot T). \label{ajut2}
\end{eqnarray}
The fact that $\alpha $ is multiplicative with respect to $\mu \circ T$ follows immediately from 
(\ref{hommultT}) and the fact that $\alpha $ is multiplicative with respect to $\mu $. Now we 
prove the Hom-associativity of $\mu \circ T$:
\begin{eqnarray*}
(\mu \circ T)\circ ((\mu \circ T)\otimes \alpha )
&\overset{(\ref{ajut1})}{=}&\mu \circ T\circ
(\mu \otimes \alpha )\circ (T\otimes id_D)\\
&\overset{(\ref{hompstw2})}{=}&\mu \circ (\mu \otimes \alpha )\circ \tilde{T}_2\circ 
(id_D\otimes T)\circ (T\otimes id_D)\\
&\overset{(\ref{hompstw3})}{=}&\mu \circ (\mu \otimes \alpha)\circ \tilde{T}_1\circ
(T\otimes id_D)\circ (id_D\otimes T)\\
&\overset{Hom-associativity\;of\;\alpha }{=}&\mu \circ (\alpha \otimes \mu )\circ \tilde{T}_1\circ 
(T\otimes id_D)\circ
(id_D\otimes T)\\
&\overset{(\ref{hompstw1})}{=}&\mu \circ T\circ (\alpha \otimes \mu )
\circ (id_D\otimes T)\\
&\overset{(\ref{ajut2})}{=}&(\mu \circ T)\circ (\alpha \otimes (\mu \circ T)),
\end{eqnarray*}
finishing the proof.
\end{proof}
\begin{definition} \label{Def-HomTwistor}
Let $(D, \mu , \alpha )$ be a Hom-associative algebra and $T:D\otimes
D\rightarrow D\otimes D$ a linear map, satisfying the following conditions:
\begin{eqnarray}
&&(\alpha \ot \alpha )\circ T=T\circ (\alpha \ot \alpha ), \label{multtwistor} \\
&&T\circ (\alpha \otimes \mu )=
(\alpha \otimes \mu )\circ T_{13}\circ T_{12},
\label{homtwistor1} \\
&&T\circ (\mu \otimes \alpha )=
(\mu \otimes \alpha )\circ T_{13}\circ T_{23},
\label{homtwistor2} \\
&&T_{12}\circ T_{23}=T_{23}\circ T_{12}.  \label{homtwistor3}
\end{eqnarray}
Such a map $T$ is called a {\em Hom-twistor}. Obviously,  a Hom-twistor $T$ is a Hom-pseudotwistor, 
with companions 
$\tilde{T}_1=\tilde{T}_2=T_{13}$, so we can consider the Hom-associative algebra
$D^T:=(D, \mu \circ T, \alpha )$. 
\end{definition}

\begin{example}\label{example23}
We consider the 2-dimensional Hom-associative algebra $(D,\mu,\alpha)$ defined with respect to a 
basis $\{e_1,e_2\}$ by
\begin{eqnarray*}
&&  \mu(e_1,e_1)=ae_1,\  \mu(e_1,e_2)= \mu(e_2,e_1)=\lambda_1 a e_1+  \lambda_2 a e_2,\   \\   &&\mu(e_2,e_2)=\frac{\lambda_1^2(1-2\lambda_2)a}{(1-\lambda_2)^2} e_1+ \frac{2 \lambda_1\lambda_2 a}{1-\lambda_2} e_2,
  \end{eqnarray*}
\begin{equation*}
�     \alpha (e_1 )=e_1, \   \alpha (e_2)=\lambda_1 e_1+\lambda_2 e_2,
\end{equation*}
where $a,\lambda_1,\lambda_2$ are  parameters in $k$, with $\lambda _2\neq 1$ and $a\neq 0$. It is 
easy to see that $D$ is an associative algebra if and only if $\lambda _2=0$.  

We provide an  example of  a Hom-twistor $T$ for $D$; it  is  defined with respect to the basis by
\begin{align*}
& T(e_1\otimes e_1)= e_1\otimes e_1,
& T(e_1\otimes e_2)=\frac{\lambda_1}{1-\lambda_2} e_1\otimes e_1,\\
& T(e_2\otimes e_1)= e_2\otimes e_1,
& T(e_2\otimes e_2)=\frac{\lambda_1}{1-\lambda_2} e_2\otimes e_1.
\end{align*}
By  Definition \ref{Def-HomTwistor}, we have the new Hom-associative algebra 
$D^{T}=(D, \mu _T=\mu \circ T, \alpha )$ whose multiplication is defined on the basis by 
\begin{align*}
&  \mu_T(e_1,e_1)=ae_1,\ &&   \mu_T(e_1,e_2)=  \frac{ \lambda_1a}{1-\lambda_2} e_1,\\&
\mu_T(e_2,e_1)=\lambda_1 a e_1+  \lambda_2 a e_2,\   &&\mu_T(e_2,e_2)=\frac{\lambda_1a}{1-\lambda_2} (\lambda_1e_1+ \lambda_2 e_2).
  \end{align*}
  Notice that the new multiplication is no longer commutative.
%\begin{equation*}
%�     \alpha (e_1 )=e_1, \   \alpha (e_2)=\lambda_1 e_1+\lambda_2 e_2,
%\end{equation*}

\end{example}
\begin{proposition}
Let $(D, \mu )$ be an associative algebra, $\alpha :D\rightarrow D$ an algebra endomorphism and 
$T:D\ot D\rightarrow D\ot D$ a pseudotwistor with companions $\tilde{T}_1$,  $\tilde{T}_2$. Consider 
the associative algebra $D^T=(D, \mu \circ T)$ and the Hom-associative algebra $D_{\alpha }=
(D, \alpha \circ \mu , \alpha )$. Assume that moreover we have $(\alpha \ot \alpha )\circ T=
T\circ (\alpha \ot \alpha )$.  
Then $T$ is a Hom-pseudotwistor for $D_{\alpha }$ with companions $\tilde{T}_1$,  $\tilde{T}_2$, 
the map $\alpha $ is an algebra endomorphism of the associative algebra $D^T$ and the 
Hom-associative algebras $(D_{\alpha })^T$ and $(D^T)_{\alpha }$ coincide. In particular, 
if $T$ is a twistor for $D$, then $T$ is a Hom-twistor for $D_{\alpha }$. 
\end{proposition}
\begin{proof}
The only nontrivial things to prove are the relations (\ref{hompstw1}) and 
(\ref{hompstw2}) with $\mu $ there 
replaced by the multiplication of $D_{\alpha }$, that is $\alpha \circ \mu $. We compute: 
\begin{eqnarray*}
(\alpha \otimes (\alpha \circ \mu ))\circ \tilde{T}_1\circ (T\otimes id_D)&=&
(\alpha \otimes \alpha )\circ (id_D\ot \mu )\circ \tilde{T}_1\circ (T\otimes id_D)\\
&\overset{(\ref{pstw1})}{=}&(\alpha \ot \alpha )\circ T\circ (id_D\ot \mu )\\
&=&T\circ (\alpha \ot \alpha )\circ (id_D\ot \mu )\\
&=&T\circ (\alpha \ot (\alpha \circ \mu )),
\end{eqnarray*}
so (\ref{hompstw1}) holds; similarly one can prove (\ref{hompstw2}). 
\end{proof}

We introduce now the Hom-analogue of twisted tensor products of algebras. 
\begin{definition}
Let $(A, \mu _A, \alpha _A)$ and $(B, \mu _B, \alpha _B)$ 
be two Hom-associative algebras. A linear map $R:B\ot A 
\rightarrow A\ot B$ is called a {\em Hom-twisting map} between $A$ and $B$ if the 
following conditions are satisfied: 
\begin{eqnarray}
&&(\alpha _A\ot \alpha _B)\circ R=R\circ (\alpha _B\ot \alpha _A), \label{homtwmap0}\\
&&R\circ (\alpha _B\ot \mu _A)=(\mu _A\ot \alpha _B)\circ (id_A\ot R)\circ (R\ot id_A),
\label{homtwmap1}\\
&&R\circ (\mu _B\ot \alpha _A)=(\alpha _A\ot \mu _B)\circ (R\ot id_B)\circ (id_B\ot R).
\label{homtwmap2}
\end{eqnarray}
\end{definition}

If we use the standard Sweedler-type notation $R(b\ot a)=a_R\ot b_R=a_r\ot b_r$, for 
$a\in A$, $b\in B$, then
the above conditions may be rewritten as:
\begin{eqnarray}
&&\alpha _A(a_R)\ot \alpha _B(b_R)=\alpha _A(a)_R\ot \alpha _B(b)_R, \label{homsweed0} \\
&&(aa')_R\ot \alpha _B(b)_R=a_Ra'_r\ot \alpha _B((b_R)_r), \label{homsweed1} \\
&&\alpha _A(a)_R\ot (bb')_R=\alpha _A((a_R)_r)\ot b_rb'_R, \label{homsweed2}
\end{eqnarray}
for all $a, a'\in A$ and $b, b'\in B$. 
\begin{proposition} \label{homttp}
Let $(A, \mu _A, \alpha _A)$ and $(B, \mu _B, \alpha _B)$ 
be two Hom-associative algebras and $R:B\ot A 
\rightarrow A\ot B$ a Hom-twisting map. Define the linear map 
$T:(A\ot B)\ot (A\ot B)\rightarrow (A\ot B)\ot (A\ot B)$, $T((a\ot b)\ot (a'\ot b'))=
(a\ot b_R)\ot (a'_R\ot b')$. Then $T$ is a Hom-twistor for the Hom-associative algebra 
$(A\ot B, \mu _{A\ot B}, \alpha _A\ot \alpha _B)$, the tensor product of $A$ and $B$. 
The Hom-associative algebra $(A\ot B)^T$ is denoted by $A\ot _RB$ and is called 
the {\em Hom-twisted tensor product} of $A$ and $B$; its multiplication is defined by 
$(a\ot b)(a'\ot b')=aa'_R\ot b_Rb'$, and the structure map is $\alpha _A\ot \alpha _B$. 
\end{proposition}
\begin{proof}
We need to prove that $T$ satisfies the conditions (\ref{multtwistor})-(\ref{homtwistor3}) for $A\ot B$. 
The condition (\ref{homtwistor3}) is trivially satisfied, (\ref{multtwistor}) follows immediately from 
(\ref{homsweed0}), while (\ref{homtwistor1}) and (\ref{homtwistor2}) follow after some easy 
computations by using (\ref{homsweed1}) and respectively (\ref{homsweed2}). 
\end{proof}
\begin{remark}
Let $(A, \mu _A, \alpha _A)$ and $(B, \mu _B, \alpha _B)$ 
be two Hom-associative algebras. Then obviously the linear map  
$R:B\ot A \rightarrow A\ot B$, $R(b\ot a)=a\ot b,$ is a Hom-twisting map and the 
Hom-twisted tensor product $A\ot _RB$ coincides with the ordinary tensor product 
$A\ot B$. 
\end{remark}
\begin{example} We assume that the characteristic of $k$ is zero 
and consider  the algebra $D$ defined in Example \ref{example23} with $\lambda _1\neq 0$ and $\lambda_2=0$. 
Recall that this $D$ is associative, but we regard it as a Hom-associative algebra with the same multiplication 
but with structure map as defined in Example \ref{example23}, that is $\alpha (e_1)=e_1$, $\alpha (e_2)=
\lambda _1e_1$. One can see that this Hom-associative algebra $D$ is a twisting, in the sense of 
Proposition \ref{yautwisting}, of the associative algebra $D$, via the map $\alpha $. 
We introduce two families of examples of Hom-twisting maps, denoted by $R_1$ and $R_2$, between 
this Hom-associative algebra and itself. They are defined with respect to the given basis by 
%\begin{small}
\begin{eqnarray*}
&& R_1(e_1\ot e_1)=0,\\
&& R_1(e_1\ot e_2)=a_1 e_1\ot e_1+a_2 e_1\ot e_2-(a_2+\frac{a_1}{\lambda_1})e_2\ot e_1,\\
&& R_1(e_2\ot e_1)=a_3 e_1\ot e_1-\frac{1}{2\lambda_1}(a_1+a_3-a_4+a_5+2a_2 \lambda_1) e_1\ot e_2\\
&&\;\;\;\;\;\;\;\;\;\;\;\;\;\;\;\;\;\;\;\;\;\;+\frac{1}{2\lambda_1}(a_1-a_3-a_4+a_5+2a_2 \lambda_1)e_2\ot e_1,\\
&& R_1(e_2\ot e_2)=\frac{\lambda_1}{2}(a_1+a_3-a_4-a_5) e_1\ot e_1+a_4 e_1\ot e_2+a_5 e_2\ot e_1\\
&&\;\;\;\;\;\;\;\;\;\;\;\;\;\;\;\;\;\;\;\;\;\;-\frac{1}{2\lambda_1}(a_1+a_3+a_4+a_5)e_2\ot e_2
\end{eqnarray*}
%\end{small}
and
%\begin{small}
\begin{eqnarray*}
&&R_2(e_1\ot e_1)=e_1\ot e_1,\\
&&R_2(e_1\ot e_2)=a_1 e_1\ot e_1+a_2 e_1\ot e_2+(1-a_2-\frac{a_1}{\lambda_1})e_2\ot e_1,\\
&&R_2(e_2\ot e_1)=a_3 e_1\ot e_1-\frac{1}{2\lambda_1}(a_1+a_3-a_4+a_5+2a_2 \lambda_1-2\lambda_1) 
e_1\ot e_2\\
&& \;\;\;\;\;\;\;\;\;\;\;\;\;\;\;\;\;\;\;\;\;\;+\frac{1}{2\lambda_1}(a_1-a_3-a_4+a_5+2a_2 \lambda_1)e_2\ot e_1,\\
&& R_2(e_2\ot e_2)=\frac{\lambda_1}{2}(a_1+a_3-a_4-a_5) e_1\ot e_1+a_4 e_1\ot e_2+a_5 e_2\ot e_1\\ 
&&\;\;\;\;\;\;\;\;\;\;\;\;\;\;\;\;\;\;\;\;\;\; -\frac{1}{2\lambda_1}(a_1+a_3+a_4+a_5-2\lambda_1)e_2\ot e_2,
\end{eqnarray*}
%\end{small}
where $a_1,\cdots , a_5$ are parameters in $k$. It is worth mentioning that in general $R_1$ and $R_2$ are 
{\em not} twisting maps for the associative algebra $D$. 
\end{example}
\begin{example}
We present now a family of examples of Hom-twisting maps between the Hom-associative algebra $D$ defined in 
Example \ref{example23}, for which we choose again $\lambda _1\neq 0$, $\lambda _2=0$, 
and the associative algebra 
$k^2$ (with a basis $\{f_1, f_2\}$ with multiplication  $f_i\cdot f_j=\delta_{ij} f_i$, for $i,j\in \{1,2\}$,  
where $\delta_{ij}$ is the Kronecker symbol), considered as a Hom-associative algebra with structure map equal to the 
identity. Namely, the twisting maps are defined with respect to the bases by the following formulae:
%\begin{small}
\begin{align*}
& R(f_1\ot e_1)=0,\\
& R(f_1\ot e_2)=a_1 e_1\ot f_1+a_2 e_1\ot f_2-\frac{a_1}{\lambda_1}e_2\ot f_1-\frac{a_2}{\lambda_1}e_2\ot f_2,\\
& R(f_2\ot e_1)=0,\\
& R(f_2\ot e_2)=a_1\lambda_1 e_1\ot f_1+a_2\lambda_1 e_1\ot f_2-a_1 e_2\ot f_1-a_2 e_2\ot f_2, 
\end{align*}
%\end{small}
where $a_1, a_2$ are parameters in $k$. Note also that in general $R$ is \emph{not} a twisting maps 
between the associative algebras $D$ and $k^2$, therefore the obtained algebra is no longer associative.
\end{example}
\begin{proposition} \label{deformttp}
Let $(A, \mu _A)$ and $(B, \mu _B)$ be two  associative algebras, $\alpha _A:A\rightarrow A$ and 
$\alpha _B:B\rightarrow B$ algebra maps and $R:B\ot A\rightarrow A\ot B$ a twisting map 
satisfying the condition $(\alpha _A\ot \alpha _B)\circ R=R\circ (\alpha _B\ot \alpha _A)$. Then $R$ is a 
Hom-twisting map between the Hom-associative algebras $A_{\alpha _A}$ and $B_{\alpha _B}$ 
and the Hom-associative algebras $A_{\alpha _A}\ot _RB_{\alpha _B}$ and 
$(A\ot _RB)_{\alpha _A\ot \alpha _B}$ coincide.
\end{proposition}
\begin{proof}
Note first that $\alpha _A\ot \alpha _B$ is an algebra endomorphism of $A\ot _RB$ because of the relation 
$(\alpha _A\ot \alpha _B)\circ R=R\circ (\alpha _B\ot \alpha _A)$. We need to prove 
(\ref{homtwmap1}) and (\ref{homtwmap2}) with those $\mu _A$ and $\mu _B$ replaced by 
$\alpha _A\circ \mu _A$ and respectively $\alpha _B\circ \mu _B$. We prove only 
(\ref{homtwmap1}), while (\ref{homtwmap2}) is similar and left to the reader: 
\begin{eqnarray*}
((\alpha _A\circ \mu _A)\ot \alpha _B)\circ (id_A\ot R)\circ (R\ot id_A)&=&
(\alpha _A\ot \alpha _B)\circ (\mu _A\ot id_B)\\
&&\circ (id_A\ot R)\circ (R\ot id_A)\\
&\overset{(\ref{twmap1})}{=}&(\alpha _A\ot \alpha _B)\circ R\circ (id_B\ot \mu _A)\\
&=&R\circ (\alpha _B\ot \alpha _A)\circ (id_B\ot \mu _A)\\
&=&R\circ (\alpha _B\ot (\alpha _A\circ \mu _A)), \;\;\;q.e.d.
\end{eqnarray*} 
The fact that the multiplications of $A_{\alpha _A}\ot _RB_{\alpha _B}$ and 
$(A\ot _RB)_{\alpha _A\ot \alpha _B}$ coincide is an immediate consequence of the relation 
$(\alpha _A\ot \alpha _B)\circ R=R\circ (\alpha _B\ot \alpha _A)$.
\end{proof}

 Let $(A, \mu _A)$ be a 
(not necessarily associative) algebra over $k$, let $q\in k$ 
be a nonzero fixed element and $\sigma :A\rightarrow A$ an involutive (i.e. $\sigma ^2=id_A$) 
algebra automorphism. We denote by $C(k, q)$ the $2$-dimensional 
associative algebra $k[v]/(v^2=q)$. Define the linear map 
\begin{eqnarray}
&&R:C(k, q)\otimes A\rightarrow A\otimes C(k, q), \;\;\;
R(1\otimes a)=a\otimes 1, \;\;R(v\otimes a)=\sigma (a)\otimes v, 
\label{RClifford}
\end{eqnarray}
for all $a\in A$. The {\em Clifford process}, as introduced in \cite{am2}, \cite{wene}, 
associates to the pair  
$(A, \sigma )$ a  
(not necessarily associative) algebra structure on $A\otimes C(k, q)$, 
with multiplication 
\begin{eqnarray}
&&(a\otimes 1+b\otimes v)(c\otimes 1+d\otimes v)=(ac+qb\sigma (d))\otimes 1
+(ad+b\sigma (c))\otimes v,  
\label{multClifford}
\end{eqnarray}
for all $a, b, c, d\in A$. This algebra structure is denoted by 
$\overline{A}$. As noted in \cite{am2}, 
if $A$ is associative then so is $\overline{A}$,  
and in this case $R$ is a twisting map and $\overline{A}$ is the 
twisted tensor product of associative algebras 
$\overline{A}=A\otimes _RC(k, q)$. If $A$ is a quasialgebra, i.e. 
$A$ is a left $H$-module algebra over a quasi-bialgebra $H$ and 
moreover $\sigma $ is $H$-linear, then $\overline{A}$ is also a 
quasialgebra, cf. \cite{ap}.

Assume now that $(A, \mu _A, \alpha _A)$ is Hom-associative and we have 
$\alpha _A\circ \sigma =\sigma \circ \alpha _A$. We can regard $B=C(k, q)$ as a 
Hom-associative algebra with $\alpha _B=id$. 
\begin{proposition}
The map $R$ defined by (\ref{RClifford}) is a Hom-twisting map and the Hom-twisted 
tensor product $A\ot _RC(k, q)$ and $\overline{A}$ are isomorphic as algebras. 
Consequently, $\overline{A}$ is a Hom-associative algebra. 
\end{proposition}
\begin{proof}
We begin with (\ref{homtwmap0}), which is enough to be checked on elements of 
the type $1\ot a$ and $v\ot a$, with $a\in A$:
\begin{eqnarray*}
((\alpha _A\ot id)\circ R)(1\ot a)&=&(\alpha _A\ot id)(a\ot 1)=
\alpha _A(a)\ot 1\\
&=&R(1\ot \alpha _A(a))=
(R\circ (id\ot \alpha _A))(1\ot a), 
\end{eqnarray*}
\begin{eqnarray*}
((\alpha _A\ot id)\circ R)(v\ot a)&=&(\alpha _A\ot id)(\sigma  (a)\ot v)=
\alpha _A(\sigma (a))\ot v=\sigma (\alpha _A(a))\ot v\\
&=&R(v\ot \alpha _A(a))=(R\circ (id\ot \alpha _A))(v\ot a).
\end{eqnarray*}
Similarly one has to check (\ref{homtwmap1}) on elements of the type $1\ot a\ot a'$ and 
$v\ot a\ot a'$ 
and (\ref{homtwmap2}) on elements of the type $1\ot 1\ot a$, $1\ot v\ot a$, $v\ot 1\ot a$ and 
$v\ot v\ot a$, with $a, a'\in A$.  Let us only check (\ref{homtwmap2}) on $v\ot v\ot a$:
\begin{eqnarray*}
&&(R\circ (\mu \ot \alpha _A))(v\ot v\ot a)=R(v^2\ot \alpha _A(a))=
R(q1\ot \alpha _A(a))=q\alpha _A(a)\ot 1,
\end{eqnarray*}
\begin{eqnarray*}
((\alpha _A\ot \mu )\circ (R\ot id)\circ (id\ot R))(v\ot v\ot a)&=&
((\alpha _A\ot \mu )\circ (R\ot id))(v\ot \sigma (a)\ot v)\\
&=&(\alpha _A\ot \mu )(\sigma ^2(a)\ot v\ot v)\\ 
&=&\alpha _A(a)\ot v^2=\alpha _A(a)\ot q1, \;\;\;q.e.d.
\end{eqnarray*}
The fact that $\overline{A}$ is exactly the Hom-twisted tensor product $A\ot _RC(k, q)$ is obvious. 
\end{proof}
\begin{remark}
In particular, if $(A, \mu _A, \alpha _A)$ is a Hom-associative algebra such that $\alpha _A^2=id_A$, 
we can perform the Clifford process with $\sigma :=\alpha _A$ to obtain the 
new Hom-associative algebra $\overline{A}$.
\end{remark}

We prove that, under certain circumstances, Hom-twisted tensor products can be iterated, generalizing thus 
the corresponding result obtained for associative algebras in \cite{jlpvo}.
\begin{theorem} \label{homiterated}
Let $(A, \mu _A, \alpha _A)$, $(B, \mu _B, \alpha _B)$ and $(C, \mu _C, \alpha _C)$ be three  
Hom-associative algebras and $R_1:B\ot A\rightarrow A\ot B$, $R_2:C\ot B\rightarrow B\ot C$, 
$R_3:C\ot A\rightarrow A\ot C$ three Hom-twisting maps, satisfying the braid condition 
\begin{eqnarray}
&&(id_A\ot R_2)\circ (R_3\ot id_B)\circ (id_C\ot R_1)=
(R_1\ot id_C)\circ (id_B\ot R_3)\circ (R_2\ot id_A). \label{hombraid}
\end{eqnarray}
Define the maps
\begin{eqnarray*}
&&P_1:C\ot (A\ot _{R_1}B)\rightarrow (A\ot _{R_1}B)\ot C, \;\;\;P_1=(id_A\ot R_2)\circ (R_3\ot id_B), \\
&&P_2:(B\ot _{R_2}C)\ot A\rightarrow A\ot (B\ot _{R_2}C), \;\;\;P_2=(R_1\ot id_C)\circ (id_B\ot R_3).
\end{eqnarray*} 
Then $P_1$ is a Hom-twisting map between $A\ot _{R_1}B$ and $C$, $P_2$ is a 
Hom-twisting map between $A$ and $B\ot _{R_2}C$, and the Hom-associative algebras 
$(A\ot _{R_1}B)\ot _{P_1}C$ and $A\ot _{P_2}(B\ot _{R_2}C)$ coincide; this Hom-associative 
algebra will be denoted by $A\ot _{R_1}B\ot _{R_2}C$ and will be called the {\em iterated Hom-twisted tensor 
product} of $A$, $B$, $C$. 
\end{theorem}
\begin{proof}
We prove that $P_1$ is a Hom-twisting map (the proof for $P_2$ is similar and left to the reader). 
We will use the Sweedler-type notation introduced before. With this notation, $P_1$ is defined by 
$P_1(c\ot a\ot b)=a_{R_3}\ot b_{R_2}\ot (c_{R_3})_{R_2}$, and (\ref{hombraid}) may be 
written as 
\begin{eqnarray}
&&(a_{R_1})_{R_3}\ot (b_{R_1})_{R_2}\ot (c_{R_3})_{R_2}=
(a_{R_3})_{R_1}\ot (b_{R_2})_{R_1}\ot (c_{R_2})_{R_3}, \label{sweedbraid}
\end{eqnarray}
for all $a\in A$, $b\in B$, $c\in C$. We prove (\ref{homtwmap0}) for $P_1$: 
\begin{eqnarray*}
((\alpha _A\ot \alpha _B\ot \alpha _C)\circ P_1)(c\ot a\ot b)&=&
\alpha _A(a_{R_3})\ot \alpha _B(b_{R_2})\ot \alpha _C((c_{R_3})_{R_2})\\
&\overset{(\ref{homsweed0})}{=}&\alpha _A(a_{R_3})\ot \alpha _B(b)_{R_2}\ot 
\alpha _C(c_{R_3})_{R_2}\\
&\overset{(\ref{homsweed0})}{=}&\alpha _A(a)_{R_3}\ot \alpha _B(b)_{R_2}\ot 
(\alpha _C(c)_{R_3})_{R_2}\\
&=&(P_1\circ (\alpha _C\ot \alpha _A\ot \alpha _B))(c\ot a\ot b), \;\;\;q.e.d.
\end{eqnarray*}
Now we prove (\ref{homtwmap1}) for $P_1$: \\[2mm]
${\;\;\;\;}$
$(P_1\circ (\alpha _C\ot \mu _{A\ot _{R_1}B}))(c\ot a\ot b\ot a'\ot b')$
\begin{eqnarray*}
&=&P_1(\alpha _C(c)\ot aa'_{R_1}\ot b_{R_1}b')\\[2mm]
&=&(aa'_{R_1})_{R_3}\ot (b_{R_1}b')_{R_2}\ot (\alpha _C(c)_{R_3})_{R_2}\\
&\overset{(\ref{homsweed1})}{=}&a_{R_3}(a'_{R_1})_{r_3}\ot (b_{R_1}b')_{R_2} 
\ot \alpha _C((c_{R_3})_{r_3})_{R_2}\\
&\overset{(\ref{homsweed1})}{=}&a_{R_3}(a'_{R_1})_{r_3}\ot (b_{R_1})_{R_2}b'_{r_2}
\ot \alpha _C((((c_{R_3})_{r_3})_{R_2})_{r_2}), 
\end{eqnarray*}
${\;\;\;\;}$
$((\mu _{A\ot _{R_1}B}\ot \alpha _C)\circ (id_A\ot id_B\ot P_1)\circ 
(P_1\ot id_A\ot id_B))(c\ot a\ot b\ot a'\ot b')$
\begin{eqnarray*}
&=&((\mu _{A\ot _{R_1}B}\ot \alpha _C)\circ (id_A\ot id_B\ot P_1))(a_{R_3}\ot 
b_{R_2}\ot (c_{R_3})_{R_2}\ot a'\ot b')\\
&=&(\mu _{A\ot _{R_1}B}\ot \alpha _C)(a_{R_3}\ot b_{R_2}\ot a'_{r_3}\ot b'_{r_2} 
\ot (((c_{R_3})_{R_2})_{r_3})_{r_2})\\
&=&a_{R_3}(a'_{r_3})_{R_1}\ot (b_{R_2})_{R_1}b'_{r_2}
\ot \alpha _C((((c_{R_3})_{R_2})_{r_3})_{r_2})\\
&\overset{(\ref{sweedbraid})}{=}&a_{R_3}(a'_{R_1})_{r_3}\ot (b_{R_1})_{R_2}b'_{r_2}
\ot \alpha _C((((c_{R_3})_{r_3})_{R_2})_{r_2}), \;\;\;q.e.d.
\end{eqnarray*}
Finally, we prove (\ref{homtwmap2}) for $P_1$: 
\begin{eqnarray*}
(P_1\circ (\mu _C\ot \alpha _A\ot \alpha _B))(c\ot c'\ot a\ot b)&=&
P_1(cc'\ot \alpha _A(a)\ot \alpha _B(b))\\
&=&\alpha _A(a)_{R_3}\ot \alpha _B(b)_{R_2}\ot ((cc')_{R_3})_{R_2}\\
&\overset{(\ref{homsweed2})}{=}&\alpha _A((a_{R_3})_{r_3})
\ot \alpha _B(b)_{R_2}\ot (c_{r_3}c'_{R_3})_{R_2}\\
&\overset{(\ref{homsweed2})}{=}&\alpha _A((a_{R_3})_{r_3})
\ot \alpha _B((b_{R_2})_{r_2})\ot (c_{r_3})_{r_2}(c'_{R_3})_{R_2}, 
\end{eqnarray*} 
${\;\;\;\;}$
$((\alpha _A\ot \alpha _B\ot \mu _C)\circ (P_1\ot id_C)\circ (id_C\ot P_1))
(c\ot c'\ot a\ot b)$
\begin{eqnarray*}
&=&((\alpha _A\ot \alpha _B\ot \mu _C)\circ (P_1\ot id_C))(c\ot a_{R_3}\ot b_{R_2}
\ot (c'_{R_3})_{R_2})\\
&=&(\alpha _A\ot \alpha _B\ot \mu _C)((a_{R_3})_{r_3}\ot (b_{R_2})_{r_2}\ot 
(c_{r_3})_{r_2}\ot (c'_{R_3})_{R_2})\\
&=&\alpha _A((a_{R_3})_{r_3})\ot \alpha _B((b_{R_2})_{r_2})\ot 
(c_{r_3})_{r_2}(c'_{R_3})_{R_2}), \;\;\;q.e.d.
\end{eqnarray*}
The fact that $(A\ot _{R_1}B)\ot _{P_1}C$ and $A\ot _{P_2}(B\ot _{R_2}C)$ coincide is 
obvious, because the multiplications in these algebras are both defined by 
$(a\ot b\ot c)(a'\ot b'\ot c')=a(a'_{R_3})_{R_1}\ot b_{R_1}b'_{R_2}\ot 
(c_{R_3})_{R_2}c'$. 
\end{proof}
%%%%%%%%%%%%%%%%%%%%%%%%%%%%%%%%%%%%%%
\section{Hom-smash products}
%%%%%%%%%%%%%%%%%%%%%%%%%%%%%
\setcounter{equation}{0}
%%%%%%%%%%%%%%%%%%%%%%%%%%%%
${\;\;\;}$We introduce now a Hom-analogue of  the smash product.
\begin{theorem}\label{smashThm}
Let $(H, \mu _H, \Delta _H, \alpha _H)$ be a Hom-bialgebra, 
$(A, \mu _A, \alpha _A)$ a left $H$-module Hom-algebra,  
with action denoted by $H\ot A\rightarrow A$, $h\ot a\mapsto h\cdot a$, and 
assume that the structure maps $\alpha _H$ and $\alpha _A$ are both bijective. 
Define the linear map 
\begin{eqnarray}
&&R:H\ot A\rightarrow A\ot H, \;\;\;R(h\ot a)=\alpha _H^{-2}(h_1)\cdot \alpha _A^{-1}(a)\ot 
\alpha _H^{-1}(h_2). \label{twmapleft}
\end{eqnarray}
Then $R$ is a Hom-twisting map between $A$ and $H$. Consequently, we can consider the 
Hom-associative algebra $A\ot _RH$, which is denoted by $A\# H$ (we denote $a\ot h:=a\# h$, 
for $a\in A$, $h\in H$) and called the {\em Hom-smash product} of $A$ and $H$. 
The structure map of $A\# H$ is $\alpha _A\ot \alpha _H$ and its multiplication is 
\begin{eqnarray*}
&&(a\# h)(a'\# h')=a(\alpha _H^{-2}(h_1)\cdot \alpha _A^{-1}(a'))\# \alpha _H^{-1}(h_2)h'.
\end{eqnarray*}
\end{theorem}
\begin{proof}
We need to prove that $R$ satisfies the conditions (\ref{homtwmap0})-(\ref{homtwmap2}). \\
\underline{Proof of  (\ref{homtwmap0})}: 
\begin{eqnarray*}
((\alpha _A\ot \alpha _H)\circ R)(h\ot a)&=&\alpha _A(\alpha _H^{-2}(h_1)\cdot \alpha _A^{-1}(a))
\ot \alpha _H(\alpha _H^{-1}(h_2))\\
&\overset{(\ref{hommod1})}{=}&\alpha _H^{-1}(h_1)\cdot a\ot h_2, 
\end{eqnarray*}
\begin{eqnarray*}
(R\circ (\alpha _H\ot \alpha _A))(h\ot a)&=&R(\alpha _H(h)\ot \alpha _A(a))\\
&=&\alpha _H^{-2}(\alpha _H(h)_1)\cdot \alpha _A^{-1}(\alpha _A(a))\ot 
\alpha _H^{-1}(\alpha _H(h)_2)\\
&\overset{(\ref{hombia3})}{=}&\alpha _H^{-2}(\alpha _H(h_1))\cdot a\ot 
\alpha _H^{-1}(\alpha _H(h_2))\\
&=&\alpha _H^{-1}(h_1)\cdot a\ot h_2, \;\;\;q.e.d.
\end{eqnarray*}
\underline{Proof of  (\ref{homtwmap1})}:
\begin{eqnarray*}
(R\circ (\alpha _H\ot \mu_ A))(h\ot a\ot a')&=&R(\alpha _H(h)\ot aa')\\
&=&\alpha _H^{-2}(\alpha _H(h)_1)\cdot \alpha _A^{-1}(aa')\ot 
\alpha _H^{-1}(\alpha _H(h)_2)\\
&\overset{(\ref{hombia3})}{=}&\alpha _H^{-1}(h_1)\cdot \alpha _A^{-1}(aa')\ot h_2, 
\end{eqnarray*}
${\;\;\;\;\;\;}$
$((\mu _A\ot \alpha _H)\circ (id_A\ot R)\circ (R\ot id_A))(h\ot a\ot a')$
\begin{eqnarray*}
&=&((\mu _A\ot \alpha _H)\circ (id_A\ot R))(\alpha _H^{-2}(h_1)\cdot \alpha _A^{-1}(a)\ot 
\alpha _H^{-1}(h_2)\ot a')\\
&=&(\mu _A\ot \alpha _H)(\alpha _H^{-2}(h_1)\cdot \alpha _A^{-1}(a)\ot 
\alpha _H^{-2}(\alpha _H^{-1}(h_2)_1)\cdot \alpha _A^{-1}(a')\ot 
\alpha _H^{-1}(\alpha _H^{-1}(h_2)_2))\\
&\overset{(\ref{hombia3})}{=}&[\alpha _H^{-2}(h_1)\cdot \alpha _A^{-1}(a)]
[\alpha _H^{-3}((h_2)_1)\cdot \alpha _A^{-1}(a')]\ot \alpha _H^{-1}((h_2)_2)\\
&\overset{(\ref{hombia1})}{=}&[\alpha _H^{-3}((h_1)_1)\cdot \alpha _A^{-1}(a)]
[\alpha _H^{-3}((h_1)_2)\cdot \alpha _A^{-1}(a')]\ot h_2\\
&\overset{(\ref{hombia3})}{=}&[\alpha _H^{-3}(h_1)_1\cdot \alpha _A^{-1}(a)]
[\alpha _H^{-3}(h_1)_2\cdot \alpha _A^{-1}(a')]\ot h_2\\
&\overset{(\ref{modalgcompat})}{=}&\alpha _H^{-1}(h_1)\cdot (\alpha _A^{-1}(a)
\alpha _A^{-1}(a'))\ot h_2\\
&=&\alpha _H^{-1}(h_1)\cdot \alpha _A^{-1}(aa') \ot h_2, \;\;\;q.e.d.
\end{eqnarray*}
\underline{Proof of  (\ref{homtwmap2})}:
\begin{eqnarray*}
(R\circ (\mu _H\ot \alpha _A))(h\ot h'\ot a)&=&
R(hh'\ot \alpha _A(a))\\
&=&\alpha _H^{-2}((hh')_1)\cdot \alpha _A^{-1}(\alpha _A(a))\ot 
\alpha _H^{-1}((hh')_2)\\
&\overset{(\ref{hombia2})}{=}&\alpha _H^{-2}(h_1h'_1)\cdot a\ot 
\alpha _H^{-1}(h_2h'_2),
\end{eqnarray*}
${\;\;\;\;\;\;}$
$((\alpha _A\ot \mu _H)\circ (R\ot id_H)\circ (id_H\ot R))(h\ot h'\ot a)$
\begin{eqnarray*}
&=&((\alpha _A\ot \mu _H)\circ (R\ot id_H))(h\ot \alpha _H^{-2}(h'_1)\cdot \alpha _A^{-1}(a)\ot 
\alpha _H^{-1}(h'_2))\\
&=&(\alpha _A\ot \mu _H)(\alpha _H^{-2}(h_1)\cdot \alpha _A^{-1}(\alpha _H^{-2}(h'_1)
\cdot \alpha _A^{-1}(a))\ot \alpha _H^{-1}(h_2)\ot \alpha _H^{-1}(h'_2))\\
&\overset{(\ref{hommod1})}{=}&\alpha _A(\alpha _H^{-2}(h_1)\cdot 
[\alpha _H^{-3}(h'_1)\cdot \alpha _A^{-2}(a)])\ot \alpha _H^{-1}(h_2h'_2)\\
&\overset{(\ref{hommod2})}{=}&\alpha _A([\alpha _H^{-3}(h_1)
\alpha _H^{-3}(h'_1)]\cdot \alpha _A^{-1}(a))\ot \alpha _H^{-1}(h_2h'_2)\\
&\overset{(\ref{hommod1})}{=}&\alpha _H^{-2}(h_1h'_1)\cdot a\ot 
\alpha _H^{-1}(h_2h'_2),
\end{eqnarray*}
finishing the proof.
\end{proof}
\begin{example}
We consider  the class of examples of $U_q(\mathfrak{sl}_2)_{\alpha }$-module Hom-algebra structures on $\mathbb{A}_{q,\beta}^{2|0}$  given in \cite[Example 5.7]{homquantum3} (here we take the base field 
$k=\mathbb{C}$). 
The quantum group $U_q(\mathfrak{sl}_2)$ is generated as a unital associative algebra by 4 generators 
$\{E,F,K,K^{-1}\}$ with relations 
\begin{eqnarray*}
&& K K^{-1}=1=K^{-1} K, \\ &&  K E=q^2 E K, \ KF= q^{-2} FK, \\ && � EF -FE =\frac{K-K^{-1}}{q-q^{-1}}
\end{eqnarray*}
where $q\in \mathbb{C}$ with $q\neq 0$, $q\neq \pm 1$.
The comultiplication is defined by 
\begin{eqnarray*}
&&  \Delta (E)=1\otimes E+ E\otimes K, \\ &&   \Delta (F)=K^{-1}\otimes F+ F\otimes 1,\\
&&  \Delta (K)=K\otimes K,\ \Delta (K^{-1})=K^{-1}\otimes K^{-1}.
 \end{eqnarray*} 
We fix $\lambda \in \mathbb{C}$, $\lambda \neq 0$. 
 The Hom-bialgebra $U_q(\mathfrak{sl}_2)_{\alpha}=(U_q(\mathfrak{sl}_2),\mu_{\alpha},\Delta_{\alpha},\alpha)$ 
is defined by $\mu_{\alpha}=\alpha\circ \mu$ and $\Delta_{\alpha}=\Delta\circ \alpha$, where $\mu$ and $\Delta$ are respectively the multiplication and comultiplication of $U_q(\mathfrak{sl}_2)$ and  
$\alpha:U_q(\mathfrak{sl}_2)\rightarrow U_q(\mathfrak{sl}_2)$ is a bialgebra morphism such that 
\begin{eqnarray*}
&&\alpha (E)=\lambda E, \ \alpha (F)=\lambda^{-1} F, \ \alpha (K)=K, \ \alpha (K^{-1})=K^{-1}.
\end{eqnarray*}
 Let  $\mathbb{A}_{q}^{2|0}=k\langle x,y\rangle /(yx-q xy)$ be the quantum plane. We fix also some 
$\xi \in \mathbb{C}$, $\xi \neq 0$. 
The Hom-quantum plane  $\mathbb{A}_{q,\beta}^{2|0}=(\mathbb{A}_{q}^{2|0},\mu_\beta,\beta)$ is defined by $\mu_\beta=\beta\circ \mu_\mathbb{A}$, where $\mu_\mathbb{A}$ is the multiplication in 
$\mathbb{A}_{q}^{2|0}$ and $\beta: \mathbb{A}_{q}^{2|0}\rightarrow \mathbb{A}_{q}^{2|0} $ 
is an algebra morphism such that $\beta (x)=\xi x,\ \beta (y)= \xi \lambda^{-1} y$. 
Then for any integer $l\geq 0$ there is a   $U_q(\mathfrak{sl}_2)_\alpha$-module Hom-algebra structure on $\mathbb{A}_{q,\beta}^{2|0}$  defined by
\begin{eqnarray*}
&& \rho_l (E, x^my^n )=[n]_q\xi^{m+n}\lambda^{l-n+1} x^{m+1}y^{n-1}\\
&& \rho_l (F, x^my^n )=[m]_q\xi^{m+n}\lambda^{-l-n-1} x^{m-1}y^{n+1} \\
&& \rho_l (K^{\pm 1}, P )=P(q^{\pm 1}\xi x,q^{\mp 1} \xi\lambda^{-1}y),
\end{eqnarray*}
where $P=P(x,y)\in \mathbb{A}_{q}^{2|0}$ and $[n]_q=\frac{q^n-q^{-n}}{q-q^{-1}}$.
 
 Notice that both $\alpha$ and $\beta$ are bijective for $\lambda\neq 0$ and $\xi \neq 0$. According to 
 Theorem \ref{smashThm}, the map $R:U_q(\mathfrak{sl}_2)_\alpha\otimes \mathbb{A}_{q,\beta}^{2|0}
\rightarrow \mathbb{A}_{q,\beta}^{2|0}\otimes U_q(\mathfrak{sl}_2)_\alpha $ defined in \eqref{twmapleft} 
leads to a smash product $\mathbb{A}_{q,\beta}^{2|0}\# U_q(\mathfrak{sl}_2)_\alpha $ whose 
multiplication is defined by   
 \begin{eqnarray*}
&&(a\# h)(a'\# h')=a(\alpha^{-2}(h_1)\cdot \beta^{-1}(a'))\# \alpha^{-1}(h_2)h'.
\end{eqnarray*}
 In particular, if we choose $l=0$, then 
for any $G\in U_q(\mathfrak{sl}_2)$ and $m,n,r,s\in \mathbb{N}$ we have
\begin{eqnarray*}
&&(x^my^n\# K^{\pm 1})(x^ry^s\# G)=q^{\pm r\mp s+n r}\xi^{m+n+r+s}\lambda^{-n-s}x^{m+r}y^{n+s}
\# K ^{\pm 1}\alpha(G),\\
&&(x^my^n\# E)(x^ry^s\# G)=q^{n r}\xi^{m+n+r+s}\lambda^{-n-s+1}x^{m+r}y^{n+s}\# E \alpha(G)\\ 
&& \;\;\;\;\;\;\;\;\;\;\;\;\;\;\;\;\;\;\;\;\;\;\;\;\;\;\;\;\;\;\;\;\;\;\;\;\; +[s]_q q^{n (r+1)} 
\xi ^{m+n+r+s}\lambda^{-n-s+1} x^{m+r+1}y^{n+s-1}\# K \alpha(G),\\
&&(x^my^n\# F)(x^ry^s\# G)=q^{s-r+n r}\xi^{m+n+r+s}\lambda^{-n-s-1}x^{m+r}y^{n+s}\# F \alpha(G)\\
 && \;\;\;\;\;\;\;\;\;\;\;\;\;\;\;\;\;\;\;\;\;\; \;\;\;\;\;\;\;\;\;\;\;\;\;\;\;\;+[r]_q q^{n (r-1)} \xi^{m+n+r+s}
\lambda^{-n-s-1}x^{m+r-1}y^{n+s+1}\#  \alpha(G), 
\end{eqnarray*}
where $K \alpha(G),\ E \alpha(G)$ and $F \alpha(G)$ are multiplications in $U_q(\mathfrak{sl}_2)$.
\end{example}
\begin{proposition}
In the hypotheses of and with notation as in Proposition \ref{deformmodalg}, and 
assuming moreover that the maps $\alpha _H$ and $\alpha _A$ are both bijective, 
if we denote by 
$A\# H$ the usual smash product between $A$ and $H$, then $\alpha _A\ot \alpha _H$ is an 
algebra endomorphism of $A\# H$ and the Hom-associative algebras 
$(A\# H)_{\alpha _A\ot \alpha _H}$ and $A_{\alpha _A}\# H_{\alpha _H}$ coincide.  
\end{proposition}
\begin{proof}
We recall that the multiplication of $A\# H$ is defined by $(a\# h)(a'\# h')=a(h_1\cdot a')\# h_2h'$, 
and the smash product $A\# H$ is the twisted tensor product $A\ot _PH$, where $P$ is the 
twisting map $P:H\ot A\rightarrow A\ot H$,  $P(h\ot a)=h_1\cdot a\ot h_2$. By using the condition 
$\alpha _A(h\cdot a)=\alpha _H(h)\cdot \alpha _A(a)$, one can prove immediately 
that we have $(\alpha _A\ot \alpha _H)\circ P=P\circ (\alpha _H\ot \alpha _A)$. We are 
thus in the hypotheses of  Proposition \ref{deformttp}, so the Hom-associative algebras 
$(A\# H)_{\alpha _A\ot \alpha _H}$ and $A_{\alpha _A}\ot _PH_{\alpha _H}$ coincide. 
So, the proof will be finished if we show that the Hom-twisting map $R$ affording 
the Hom-smash product $A_{\alpha _A}\# H_{\alpha _H}$ and the map $P$ actually coincide. 
We compute this map $R$ (using the structures of $A_{\alpha _A}$ and $H_{\alpha _H}$): 
\begin{eqnarray*}
R(h\ot a)&=&\alpha _H^{-2}(\alpha _H(h)_1)\triangleright \alpha _A^{-1}(a)\ot 
\alpha _H^{-1}(\alpha _H(h)_2)\\
&=& \alpha _H^{-1}(h_1)\triangleright \alpha _A^{-1}(a)\ot h_2\\
&=&h_1\cdot a\ot h_2=P(h\ot a), 
\end{eqnarray*}
so indeed we have $R=P$.   
\end{proof}

We will need in what follows the right-handed and two-sided analogues of left comodules 
and comodule algebras over Hom-coassociative coalgebras and Hom-bialgebras. 
\begin{definition}  Let $(C, \Delta _C , \alpha _C)$ be a Hom-coassociative coalgebra, 
$M$ a linear space and $\alpha _M:M\rightarrow M$ a linear map. \\
(i) A {\em right $C$-comodule} structure on $(M, \alpha _M)$ consists of a linear map 
$\rho :M\rightarrow M\ot C$ satisfying the following conditions:
\begin{eqnarray}
&&(\alpha _M\ot \alpha _C)\circ \rho =\rho \circ \alpha _M,  
\label{rightcom1}\\
&&(\alpha _M\ot \Delta _C)\circ \rho =(\rho \ot \alpha _C)\circ \rho . 
\label{rightcom2}
\end{eqnarray} 
(ii) If $(M, \alpha _M)$ is both a left $C$-comodule with structure $\lambda :M\rightarrow C\ot M$ 
and a right $C$-comodule with structure $\rho :M\rightarrow M\ot C$, then $M$ is called 
a {\em $C$-bicomodule} if 
$(\lambda \ot \alpha _C)\circ \rho =(\alpha _C\ot \rho )\circ \lambda $. 

Obviously, $(C, \alpha _C)$ itself is a $C$-bicomodule, with $\rho =\lambda =\Delta _C$. 
\end{definition}
\begin{definition}
Let $(H, \mu _H, \Delta _H, \alpha _H)$ be a Hom-bialgebra. \\
(i) A {\em right $H$-comodule Hom-algebra} 
is a Hom-associative algebra   
$(D, \mu _D, \alpha _D)$ endowed with a right $H$-comodule structure 
$\rho _D:D\rightarrow D\ot H$ such that $\rho _D$ is a morphism of Hom-associative algebras. \\
(ii) An {\em $H$-bicomodule Hom-algebra} is a Hom-associative algebra $(D, \mu _D, \alpha _D)$ 
that is both a left and a right $H$-comodule Hom-algebra and such that the left and right 
$H$-comodule structures form an $H$-bicomodule. 
\end{definition}
\begin{proposition} \label{smrca}
Let $(H, \mu _H, \Delta _H, \alpha _H)$ be a Hom-bialgebra and  
$(A, \mu _A, \alpha _A)$ a left $H$-module Hom-algebra, with action denoted by 
$H\ot A\rightarrow A$, $h\ot a\mapsto h\cdot a$, such that the structure maps 
$\alpha _H$ and $\alpha _A$ are both bijective. Then the Hom-smash product 
$A\# H$ is a right $H$-comodule Hom-algebra, via the linear map 
$\rho _{A\# H}:A\# H\rightarrow (A\# H)\ot H$, 
$\rho _{A\# H}(a\# h)=(\alpha _A(a)\# h_1)\ot h_2$. 
\end{proposition}
\begin{proof}
First we prove (\ref{rightcom2}) for the map $\rho _{A\# H}$: 
\begin{eqnarray*}
((\alpha _A\ot \alpha _H\ot \Delta _H)\circ \rho _{A\# H})(a\# h)&=&
\alpha _A^2(a)\# \alpha _H(h_1)\ot (h_2)_1\ot (h_2)_2\\
&\overset{(\ref{hombia1})}{=}&\alpha _A^2(a)\# (h_1)_1\ot (h_1)_2\ot \alpha _H(h_2)\\
&=&((\rho _{A\# H}\ot \alpha _H)\circ \rho _{A\# H})(a\# h), \;\;\;q.e.d.
\end{eqnarray*}
The relation $(\alpha _A\ot \alpha _H\ot \alpha _H)\circ \rho _{A\# H}=
\rho _{A\# H}\circ (\alpha _A\ot \alpha _H)$ follows immediately from (\ref{hombia3}), 
so the only thing left to prove is the multiplicativity of $\rho _{A\# H}$; we compute:
\begin{eqnarray*}
\rho _{A\# H}((a\# h)(a'\# h'))&=&\rho _{A\# H}(a(\alpha _H^{-2}(h_1)\cdot \alpha _A^{-1}(a'))
\# \alpha _H^{-1}(h_2)h')\\
&=&\alpha _A(a(\alpha _H^{-2}(h_1)\cdot \alpha _A^{-1}(a')))\# (\alpha _H^{-1}(h_2)h')_1\ot 
(\alpha _H^{-1}(h_2)h')_2\\
&\overset{(\ref{hommod1}), (\ref{hombia2})}{=}&\alpha _A(a)(\alpha _H^{-1}(h_1)\cdot a')
\# \alpha _H^{-1}(h_2)_1h'_1\ot  \alpha _H^{-1}(h_2)_2h'_2\\
&\overset{(\ref{hombia3})}{=}&\alpha _A(a)(\alpha _H^{-1}(h_1)\cdot a')
\# \alpha _H^{-1}((h_2)_1)h'_1\ot  \alpha _H^{-1}((h_2)_2)h'_2\\
&\overset{(\ref{hombia1})}{=}&\alpha _A(a)(\alpha _H^{-2}((h_1)_1)\cdot a')
\# \alpha _H^{-1}((h_1)_2)h'_1\ot  h_2h'_2, 
\end{eqnarray*}
\begin{eqnarray*}
\rho _{A\# H}(a\# h)\rho _{A\# H}(a'\# h')&=&((\alpha _A(a)\# h_1)\ot h_2)
((\alpha _A(a')\# h'_1)\ot h'_2)\\
&=&(\alpha _A(a)\# h_1)(\alpha _A(a')\# h'_1)\ot h_2h'_2\\
&=&\alpha _A(a)(\alpha _H^{-2}((h_1)_1)\cdot a')
\# \alpha _H^{-1}((h_1)_2)h'_1\ot  h_2h'_2, 
\end{eqnarray*}
and obviously the two terms are equal. 
\end{proof}
\begin{definition} (\cite{PanaiteMakhlouf}) 
Let $(H, \mu _H, \Delta _H, \alpha _H)$ be a Hom-bialgebra, $M$ a linear space and 
$\alpha _M:M\rightarrow M$ a linear map such that $(M, \alpha _M)$ is a left $H$-module with 
action $H\ot M\rightarrow M$, $h\ot m\mapsto h\cdot m$ and a 
left $H$-comodule with coaction $M\rightarrow H\ot M$, $m\mapsto m_{(-1)}\ot m_{(0)}$. 
Then $(M, \alpha _M)$ is called a (left-left) {\em Yetter-Drinfeld module} over $H$  if the 
following relation holds, for all $h\in H$, $m\in M$:
\begin{eqnarray}
&&(h_1\cdot m)_{(-1)}\alpha _H^2(h_2)\ot (h_1\cdot m)_{(0)}=
\alpha _H^2(h_1)\alpha _H(m_{(-1)})\ot \alpha _H(h_2)\cdot m_{(0)}. 
\label{homYD}
\end{eqnarray}
\end{definition}
\begin{lemma}
Let $(A, \mu , \alpha )$ be a Hom-associative algebra such that $\alpha $ is bijective and let $a, b, c, d\in A$. 
Then the following relation holds: 
\begin{eqnarray}
&&(ab)(cd)=\alpha (a)(\alpha ^{-1}(bc)d). \label{4elem}
\end{eqnarray}
\end{lemma}
\begin{proof}
A straightforward computation, using the definition of a Hom-associative algebra. 
\end{proof}
\begin{proposition}
In the hypotheses of and with notation as in Proposition \ref{smrca}, assume that moreover 
$(A, \alpha _A)$ is a left $H$-comodule with structure $A\rightarrow H\ot A$, 
$a\mapsto a_{(-1)}\ot a_{(0)}$, such that $(A, \mu _A, \alpha _A)$ is a left $H$-comodule Hom-algebra 
and $(A, \alpha _A)$ is a (left-left) Yetter-Drinfeld module  over $H$. Then $A\# H$ is an 
$H$-bicomodule Hom-algebra, via the map $\rho _{A\# H}$ defined in Proposition \ref{smrca} and the linear 
map $\lambda _{A\# H}:A\# H\rightarrow H\ot (A\# H)$, $\lambda _{A\# H}(a\# h)=
a_{(-1)}h_1\ot (a_{(0)}\# h_2)$. 
\end{proposition}
\begin{proof}
The fact that  $\lambda _{A\# H}$ is a left $H$-comodule structure follows from \cite{homquantum3}, 
Prop. 5.3 ($\lambda _{A\# H}$ is just the tensor product of the left $H$-comodules $A$ and $H$). 
We check the bicomodule condition:
\begin{eqnarray*}
((\lambda _{A\# H}\ot \alpha _H)\circ \rho _{A\# H})(a\# h)&=&
(\lambda _{A\# H}\ot \alpha _H)((\alpha _A(a)\# h_1)\ot h_2)\\
&=&\alpha _A(a)_{(-1)}(h_1)_1\ot (\alpha _A(a)_{(0)}\# (h_1)_2)\ot \alpha _H(h_2)\\
&\overset{(\ref{hombia1}), \;(\ref{leftcom1})}{=}&
\alpha _H(a_{(-1)})\alpha _H(h_1)\ot (\alpha _A(a_{(0)})\# (h_2)_1)\ot (h_2)_2\\
&=&(\alpha _H\ot \rho _{A\# H})(a_{(-1)}h_1\ot (a_{(0)}\# h_2))\\
&=&((\alpha _H\ot \rho _{A\# H})\circ \lambda _{A\# H})(a\# h), \;\;\;q.e.d.
\end{eqnarray*}
The only thing left to prove is that $\lambda _{A\# H}$ is multiplicative. We compute:\\[2mm]
${\;\;\;}$
$\lambda _{A\# H}((a\# h)(a'\# h'))$
\begin{eqnarray*}
&=&\lambda (a(\alpha _H^{-2}(h_1)\cdot \alpha _A^{-1}(a'))
\# \alpha _H^{-1}(h_2)h')\\
&=&[a_{(-1)}(\alpha _H^{-2}(h_1)\cdot \alpha _A^{-1}(a'))_{(-1)}] 
[\alpha _H^{-1}(h_2)_1h'_1]
\ot a_{(0)}(\alpha _H^{-2}(h_1)\cdot \alpha _A^{-1}(a'))_{(0)}\# 
\alpha _H^{-1}(h_2)_2h'_2\\
&=&[a_{(-1)}(\alpha _H^{-2}(h_1)\cdot \alpha _A^{-1}(a'))_{(-1)}] 
[\alpha _H^{-1}((h_2)_1)h'_1]
\ot a_{(0)}(\alpha _H^{-2}(h_1)\cdot \alpha _A^{-1}(a'))_{(0)}\\
&&\# 
\alpha _H^{-1}((h_2)_2)h'_2\\
&\overset{(\ref{hombia1})}{=}&
[a_{(-1)}(\alpha _H^{-3}((h_1)_1)\cdot \alpha _A^{-1}(a'))_{(-1)}] 
[\alpha _H^{-1}((h_1)_2)h'_1]
\ot a_{(0)}(\alpha _H^{-3}((h_1)_1)\cdot \alpha _A^{-1}(a'))_{(0)}\# 
h_2h'_2\\
&\overset{(\ref{4elem})}{=}&\alpha _H(a_{(-1)})
\{\alpha _H^{-1}[(\alpha _H^{-3}((h_1)_1)\cdot \alpha _A^{-1}(a'))_{(-1)}
\alpha _H^{-1}((h_1)_2)]h'_1\}\\
&&\ot a_{(0)}(\alpha _H^{-3}((h_1)_1)\cdot \alpha _A^{-1}(a'))_{(0)}\# 
h_2h'_2\\
&=&\alpha _H(a_{(-1)})
\{\alpha _H^{-1}[(\alpha _H^{-3}((h_1)_1)\cdot \alpha _A^{-1}(a'))_{(-1)}
\alpha _H^2(\alpha _H^{-3}((h_1)_2))]h'_1\}\\
&&\ot a_{(0)}(\alpha _H^{-3}((h_1)_1)\cdot \alpha _A^{-1}(a'))_{(0)}\# 
h_2h'_2\\
&=&\alpha _H(a_{(-1)})
\{\alpha _H^{-1}[(\alpha _H^{-3}(h_1)_1\cdot \alpha _A^{-1}(a'))_{(-1)}\alpha _H^2
(\alpha _H^{-3}(h_1)_2)]h'_1\}\\
&&\ot a_{(0)}(\alpha _H^{-3}(h_1)_1\cdot \alpha _A^{-1}(a'))_{(0)}\# 
h_2h'_2\\
&\overset{(\ref{homYD})}{=}&\alpha _H(a_{(-1)})
\{\alpha _H^{-1}[\alpha _H^2(\alpha _H^{-3}(h_1)_1)\alpha _H(\alpha _A^{-1}(a')_{(-1)})]h'_1\}\\
&&\ot a_{(0)}(\alpha _H(\alpha _H^{-3}(h_1)_2)\cdot \alpha _A^{-1}(a')_{(0)})\# 
h_2h'_2\\
&\overset{(\ref{leftcom1})}{=}&\alpha _H(a_{(-1)})
[\alpha _H^{-1}(\alpha _H^{-1}(h_1)_1a'_{(-1)})h'_1]
\ot a_{(0)}(\alpha _H^{-1}(\alpha _H^{-1}(h_1)_2)\cdot \alpha _A^{-1}(a'_{(0)}))\# 
h_2h'_2\\
&\overset{(\ref{4elem})}{=}&[a_{(-1)}\alpha _H^{-1}((h_1)_1)][a'_{(-1)}h'_1]\ot 
a_{(0)}(\alpha _H^{-2}((h_1)_2)\cdot \alpha _A^{-1}(a'_{(0)}))\# 
h_2h'_2\\
&\overset{(\ref{hombia1})}{=}&[a_{(-1)}h_1][a'_{(-1)}h'_1]\ot 
a_{(0)}(\alpha _H^{-2}((h_2)_1)\cdot \alpha _A^{-1}(a'_{(0)}))\# 
\alpha _H^{-1}((h_2)_2)h'_2\\
&=&[a_{(-1)}h_1][a'_{(-1)}h'_1]\ot 
(a_{(0)}\# h_2)(a'_{(0)}\# h'_2)\\
&=&\lambda _{A\# H}(a\# h)\lambda _{A\# H}(a'\# h'),
\end{eqnarray*}
finishing the proof.
\end{proof}

In order to define two-sided Hom-smash products, we need first the right-handed 
versions of some concepts and results presented so far; the proofs of these analogues 
are left to the reader.  
\begin{definition} 
Let $(A, \mu _A , \alpha _A)$ be a Hom-associative algebra, $M$ a linear space and $\alpha _M:M
\rightarrow M$ a linear map. A {\em right $A$-module} structure on $(M, \alpha _M)$ consists of a linear map 
$M\ot A\rightarrow M$, $m\ot a\mapsto m\cdot a$, satisfying the conditions:
\begin{eqnarray}
&&\alpha _M(m\cdot a)=\alpha _M(m)\cdot \alpha _A(a), \label{righthommod1}\\
&&(m\cdot a)\cdot \alpha _A(a')=\alpha _M(m)\cdot (aa'), \label{righthommod2}
\end{eqnarray} 
for all $a, a'\in A$ and $m\in M$. If $(M, \alpha _M)$ and $(N, \alpha _N)$ are right $A$-modules (both 
$A$-actions denoted by $\cdot$),  
a morphism of right $A$-modules $f:M\rightarrow N$ is a linear map satisfying the conditions 
$\alpha _N\circ f=f\circ \alpha _M$  and $f(m\cdot a)=f(m)\cdot a$, for all $a\in A$ and $m\in M$. 
\end{definition}
\begin{definition} 
Assume that $(H, \mu _H, \Delta _H, \alpha _H)$ is a Hom-bialgebra. A Hom-associative algebra 
$(C, \mu _C, \alpha _C)$ is called a {\em right $H$-module Hom-algebra} if $(C, \alpha _C)$ is a 
right $H$-module, 
with action denoted by $C\ot H\rightarrow C$, $c\ot h\mapsto c\cdot h$, such that the following 
condition is satisfied: 
\begin{eqnarray}
&&(cc')\cdot \alpha _H^2(h)=(c\cdot h_1)(c'\cdot h_2), \;\;\;\forall \;h\in H, \;c, c'\in C. 
\end{eqnarray} 
\end{definition}
\begin{proposition}  \label{rightdefmodalg}
Let $(H, \mu _H, \Delta _H)$ be a bialgebra and $(C, \mu _C)$ a right $H$-module algebra in the usual sense, 
with action denoted by $C\ot H\rightarrow C$, $c\ot h\mapsto c\cdot h$. Let $\alpha _H:H\rightarrow H$ 
be a bialgebra endomorphism and $\alpha _C:C\rightarrow C$ an algebra endomorphism, such that 
$\alpha _C(c\cdot h)=\alpha _C(c)\cdot \alpha _H(h)$, for all $h\in H$ and $c\in C$. Then 
the Hom-associative algebra $C_{\alpha _C}=(C, \alpha _C\circ \mu _C, \alpha _C)$  
becomes a right module Hom-algebra over the Hom-bialgebra 
$H_{\alpha _H}=(H, \alpha _H\circ \mu _H, \Delta _H\circ \alpha _H, \alpha _H)$, 
with action defined by  
$C_{\alpha _C}\ot H_{\alpha _H}\rightarrow C_{\alpha _C}$, $c\ot h\mapsto c\triangleleft h:=
\alpha _C(c\cdot h)=\alpha _C(c)\cdot \alpha _H(h)$. 
\end{proposition}
\begin{theorem}
Let $(H, \mu _H, \Delta _H, \alpha _H)$ be a Hom-bialgebra, 
$(C, \mu _C, \alpha _C)$ a right $H$-module Hom-algebra,  
with action denoted by $C\ot H\rightarrow C$, $c\ot h\mapsto c\cdot h$, and 
assume that the structure maps $\alpha _H$ and $\alpha _C$ are both bijective. 
Define the linear map 
\begin{eqnarray}
&&R:C\ot H\rightarrow H\ot C, \;\;\;R(c\ot h)=\alpha _H^{-1}(h_1) \ot \alpha _C^{-1}(c)\cdot 
\alpha _H^{-2}(h_2).
\label{twmapright}
\end{eqnarray}
Then $R$ is a Hom-twisting map between $H$ and $C$. Consequently, we can consider the 
Hom-associative algebra $H\ot _RC$, which is denoted by $H\# C$ (we denote $h\ot c:=h\# c$, 
for $c\in C$, $h\in H$) and called the {\em Hom-smash product} of $H$ and $C$.  
The structure map of $H\# C$ is $\alpha _H\ot \alpha _C$ and its multiplication is 
\begin{eqnarray*}
&&(h\# c)(h'\# c')=h\alpha _H^{-1}(h'_1)\# (\alpha _C^{-1}(c)\cdot \alpha _H^{-2}(h'_2))c'. 
\end{eqnarray*}
\end{theorem}
\begin{proposition}
In the hypotheses of and with notation as in Proposition \ref{rightdefmodalg}, and 
assuming moreover that the maps $\alpha _H$ and $\alpha _C$ are both bijective, 
if we denote by 
$H\# C$ the usual smash product between $H$ and $C$ (whose multiplication is 
$(h\# c)(h'\# c')=hh'_1\# (c\cdot h'_2)c'$), 
then $\alpha _H\ot \alpha _C$ is an 
algebra endomorphism of $H\# C$ and the Hom-associative algebras 
$(H\# C)_{\alpha _H\ot \alpha _C}$ and $H_{\alpha _H}\# C_{\alpha _C}$ coincide.  
\end{proposition}
\begin{proposition}
Let $(H, \mu _H, \Delta _H, \alpha _H)$ be a Hom-bialgebra and  
$(C, \mu _C, \alpha _C)$ a right $H$-module Hom-algebra, with action denoted by 
$C\ot H\rightarrow C$, $c\ot h\mapsto c\cdot h$, such that the structure maps 
$\alpha _H$ and $\alpha _C$ are both bijective. Then the Hom-smash product 
$H\# C$ is a left $H$-comodule Hom-algebra, via the map 
$\lambda _{H\# C}:H\# C\rightarrow H\ot (H\# C)$, 
$\lambda _{H\# C}(h\# c)=h_1\ot (h_2\# \alpha _C(c))$. 
\end{proposition}

We are now in the position to define the two-sided Hom-smash product, as a particular case 
of an iterated Hom-twisted tensor product. 
\begin{proposition}
Let $(H, \mu _H, \Delta _H, \alpha _H)$ be a Hom-bialgebra, $(A, \mu _A, \alpha _A)$ a left 
$H$-module Hom-algebra and 
$(C, \mu _C, \alpha _C)$ a right $H$-module Hom-algebra, with actions denoted by 
$H\ot A\rightarrow A$, $h\ot a\mapsto h\cdot a$ and 
$C\ot H\rightarrow C$, $c\ot h\mapsto c\cdot h$, and assume that the structure maps 
$\alpha _H$, $\alpha _A$, $\alpha _C$ are bijective. 
Consider the Hom-twisting maps 
defined by (\ref{twmapleft}) and (\ref{twmapright}), namely 
\begin{eqnarray*}
&&R_1:H\ot A\rightarrow A\ot H, \;\;\;R_1(h\ot a)=\alpha _H^{-2}(h_1)\cdot \alpha _A^{-1}(a)\ot 
\alpha _H^{-1}(h_2), \\
&&R_2:C\ot H\rightarrow H\ot C, \;\;\;R_2(c\ot h)=\alpha _H^{-1}(h_1) \ot \alpha _C^{-1}(c)\cdot 
\alpha _H^{-2}(h_2), 
\end{eqnarray*}
as well as the trivial Hom-twisting map $R_3:C\ot A\rightarrow A\ot C$, $R_3(c\ot a)=a\ot c$. Then 
$R_1$, $R_2$, $R_3$ satisfy the braid relation, so, by Theorem \ref{homiterated}, we can consider 
the iterated Hom-twisted tensor product $A\ot _{R_1}H\ot _{R_2}C$, which will be denoted by 
$A\# H\# C$ and will be called the {\em two-sided Hom-smash product}. Its multiplication is defined 
by 
\begin{eqnarray*}
&&(a\# h\# c)(a'\# h'\# c')=a(\alpha _H^{-2}(h_1)\cdot \alpha _A^{-1}(a'))\# 
\alpha _H^{-1}(h_2h'_1)\#  (\alpha _C^{-1}(c)\cdot \alpha _H^{-2}(h'_2))c', 
\end{eqnarray*} 
and its structure map is $\alpha _A\ot \alpha _H\ot \alpha _C$. 
\end{proposition}
\begin{proof}
We only need to prove the braid relation. We compute:\\[2mm]
${\;\;\;\;}$
$((id_A\ot R_2)\circ (R_3\ot id_H)\circ (id_C\ot R_1))(c\ot h\ot a)$
\begin{eqnarray*}
&=&((id_A\ot R_2)\circ (R_3\ot id_H))(c\ot \alpha _H^{-2}(h_1)\cdot \alpha _A^{-1}(a)\ot 
\alpha _H^{-1}(h_2))\\
&=&(id_A\ot R_2)(\alpha _H^{-2}(h_1)\cdot \alpha _A^{-1}(a)\ot c\ot 
\alpha _H^{-1}(h_2))\\
&=&\alpha _H^{-2}(h_1)\cdot \alpha _A^{-1}(a)\ot \alpha _H^{-1}(\alpha _H^{-1}(h_2)_1)
\ot \alpha _C^{-1}(c)\cdot \alpha _H^{-2}(\alpha _H^{-1}(h_2)_2)\\
&\overset{(\ref{hombia3})}{=}&\alpha _H^{-2}(h_1)\cdot \alpha _A^{-1}(a)\ot 
\alpha _H^{-2}((h_2)_1)
\ot \alpha _C^{-1}(c)\cdot \alpha _H^{-3}((h_2)_2), 
\end{eqnarray*}
${\;\;\;\;}$
$((R_1\ot id_C)\circ (id_H\ot R_3)\circ (R_2\ot id_A))(c\ot h\ot a)$
\begin{eqnarray*}
&=&((R_1\ot id_C)\circ (id_H\ot R_3))(\alpha _H^{-1}(h_1)\ot \alpha _C^{-1}(c)\cdot 
\alpha _H^{-2}(h_2)\ot a)\\
&=&(R_1\ot id_C)(\alpha _H^{-1}(h_1)\ot a\ot \alpha _C^{-1}(c)\cdot 
\alpha _H^{-2}(h_2))\\
&=&\alpha _H^{-2}(\alpha _H^{-1}(h_1)_1)\cdot \alpha _A^{-1}(a)\ot 
\alpha _H^{-1}(\alpha _H^{-1}(h_1)_2)\ot \alpha _C^{-1}(c)\cdot \alpha _H^{-2}(h_2)\\
&\overset{(\ref{hombia3})}{=}&\alpha _H^{-3}((h_1)_1)\cdot \alpha _A^{-1}(a)\ot 
\alpha _H^{-2}((h_1)_2)\ot \alpha _C^{-1}(c)\cdot \alpha _H^{-2}(h_2)\\
&\overset{(\ref{hombia1})}{=}&\alpha _H^{-2}(h_1)\cdot \alpha _A^{-1}(a)\ot 
\alpha _H^{-2}((h_2)_1)\ot \alpha _C^{-1}(c)\cdot \alpha _H^{-3}((h_2)_2), 
\end{eqnarray*}
finishing the proof.
\end{proof}
%%%%%%%%%%%%%%%%%%%%%%%%%%%%%%%%%%%%%%
\section{Hom-associative algebras obtained from associative algebras}
%%%%%%%%%%%%%%%%%%%%%%%%%%%%%
\setcounter{equation}{0}
%%%%%%%%%%%%%%%%%%%%%%%%%%%%
${\;\;\;}$Our aim now is to show that two procedures recalled in the Preliminaries, the one of twisting an 
associative algebra by a pseudotwistor to obtain another associative algebra and Yau's procedure 
of twisting an associative algebra by an endomorphism to obtain a Hom-associative algebra 
admit a common generalization. 
\begin{theorem} \label{alphapseudo}
Let $(A, \mu )$ be an associative algebra, $\alpha :A\rightarrow A$ an algebra endomorphism, 
$T:A\ot A\rightarrow A\ot A$ and $\tilde{T}_1,
\tilde{T}_2:A\otimes A\otimes A \rightarrow A\otimes A\otimes A$ linear maps, satisfying the 
following conditions: 
\begin{eqnarray}
&&(\alpha \ot \alpha )\circ T=T\circ (\alpha \ot \alpha ),  \label{multT} \\
&&T\circ (id_A\otimes \mu )=
(id_A\otimes \mu )\circ \tilde{T}_1\circ (T\otimes id_A),
\label{pstwhom1} \\
&&T\circ (\mu \otimes id_A)=
(\mu \otimes id_A)\circ \tilde{T}_2\circ (id_A\otimes T),
\label{pstwhom2} \\
&&\tilde{T}_1\circ (T\otimes id_A)\circ (\alpha \otimes T)=
\tilde{T}_2\circ (id_A\otimes T)\circ (T\otimes \alpha ). \label{pstwhom3}
\end{eqnarray}
Then $(A, \mu \circ T, \alpha )$ is a Hom-associative algebra, which is denoted by $A^T_{\alpha }$.  
The map $T$ is called an
{\em $\alpha $-pseudotwistor} and the two maps $\tilde{T}_1$,
$\tilde{T}_2$ are called the {\em companions} of $T$.
\end{theorem}
\begin{proof}
We record first the obvious relations 
\begin{eqnarray}
&&(\mu \circ T)\ot \alpha =(\mu \ot id_A )\circ (T\ot \alpha), \label{ajut3} \\
&&\alpha \ot (\mu \circ T)=(id_A \ot \mu )\circ (\alpha \ot T). \label{ajut4}
\end{eqnarray}
The fact that $\alpha $ is multiplicative with respect to $\mu \circ T$ follows immediately from 
(\ref{multT}) and the fact that $\alpha $ is multiplicative with respect to $\mu $, so we only have to  
prove the Hom-associativity of $\mu \circ T$. We compute: 
\begin{eqnarray*}
(\mu \circ T)\circ ((\mu \circ T)\otimes \alpha )
&\overset{(\ref{ajut3})}{=}&\mu \circ T\circ
(\mu \otimes id_A )\circ (T\otimes \alpha)\\
&\overset{(\ref{pstwhom2})}{=}&\mu \circ (\mu \otimes id_A )\circ \tilde{T}_2\circ 
(id_A\otimes T)\circ (T\otimes \alpha)\\
&\overset{(\ref{pstwhom3})}{=}&\mu \circ (\mu \otimes id_A)\circ \tilde{T}_1\circ
(T\otimes id_A)\circ (\alpha \otimes T)\\
&\overset{associativity\;of\;\mu }{=}&\mu \circ (id_A \otimes \mu )\circ \tilde{T}_1\circ 
(T\otimes id_A)\circ
(\alpha \otimes T)\\
&\overset{(\ref{pstwhom1})}{=}&\mu \circ T\circ (id_A \otimes \mu )
\circ (\alpha \otimes T)\\
&\overset{(\ref{ajut4})}{=}&(\mu \circ T)\circ (\alpha \otimes (\mu \circ T)),
\end{eqnarray*}
finishing the proof.
\end{proof}

Obviously, if $(A, \mu )$ is an associative algebra and we take $\alpha =id_A$, an $\alpha $-pseudotwistor 
is the same thing as a pseudotwistor and the Hom-associative algebra $A^T_{\alpha }$ is actually 
associative.

We show now that Yau's procedure is a particular case of Theorem \ref{alphapseudo}.    
\begin{proposition}
Let $(A, \mu )$ be an associative algebra and $\alpha :A\rightarrow A$ an algebra 
endomorphism. Define the maps 
\begin{eqnarray*}
&&T:A\ot A\rightarrow A\ot A, \;\;\;T=\alpha \ot \alpha, \\
&&\tilde{T}_1:A\otimes A\otimes A \rightarrow A\otimes A\otimes A, \;\;\;\tilde{T}_1=
id_A\ot id_A\ot \alpha , \\
&&\tilde{T}_2:A\otimes A\otimes A \rightarrow A\otimes A\otimes A, \;\;\; 
\tilde{T}_2=\alpha \ot id_A\ot id_A. 
\end{eqnarray*}
Then $T$ is an $\alpha $-pseudotwistor with companions $\tilde{T}_1$, $\tilde{T}_2$ and the 
Hom-associative algebras $A^T_{\alpha }$ and $A_{\alpha }$ coincide.  
\end{proposition}
\begin{proof}
The condition (\ref{multT}) is obviously satisfied. We check (\ref{pstwhom1}), for $a, b, c\in A$: 
\begin{eqnarray*}
((id_A\otimes \mu )\circ \tilde{T}_1\circ (T\otimes id_A))(a\ot b\ot c)&=&
((id_A\otimes \mu )\circ \tilde{T}_1)(\alpha (a)\ot \alpha (b)\ot c)\\
&=&(id_A\ot \mu )(\alpha (a)\ot \alpha (b)\ot \alpha (c))\\
&=&\alpha (a)\ot \alpha (b)\alpha (c)\\
&=&\alpha (a)\ot \alpha (bc)\\
&=&T(a\ot bc)\\
&=&(T\circ (id_A\ot \mu ))(a\ot b\ot c), \;\;\;q.e.d.
\end{eqnarray*}
The condition (\ref{pstwhom2}) is similar, so we check (\ref{pstwhom3}): 
\begin{eqnarray*}
(\tilde{T}_1\circ (T\otimes id_A)\circ (\alpha \otimes T))(a\ot b\ot c)&=&
(\tilde{T}_1\circ (T\otimes id_A))(\alpha (a)\ot \alpha (b)\ot \alpha (c))\\
&=&\tilde{T}_1(\alpha ^2(a)\ot \alpha ^2(b)\ot \alpha (c))\\
&=&\alpha ^2(a)\ot \alpha ^2(b)\ot \alpha ^2(c)\\
&=&\tilde{T}_2(\alpha (a)\ot \alpha ^2(b)\ot \alpha ^2(c))\\
&=&(\tilde{T}_2\circ (id_A\ot T))(\alpha (a)\ot \alpha (b)\ot \alpha (c))\\
&=&(\tilde{T}_2\circ (id_A\ot T)\circ (T\ot \alpha ))(a\ot b\ot c), \;\;\;q.e.d.
\end{eqnarray*}
The fact that $A^T_{\alpha }$ and $A_{\alpha }$ coincide is obvious. 
\end{proof}
\begin{definition}
Let $(A, \mu _A)$ and $(B, \mu _B)$ 
be two associative algebras and $\alpha _A:A\rightarrow A$ and 
$\alpha _B:B\rightarrow B$ two bijective algebra endomorphisms. 
A linear map $R:B\ot A 
\rightarrow A\ot B$ is called {\em $(\alpha _A, \alpha _B)$-twisting map}  if the 
following conditions are satisfied: 
\begin{eqnarray}
&&(\alpha _A\ot \alpha _B)\circ R=R\circ (\alpha _B\ot \alpha _A), \label{ashom0} \\
&&R\circ (id_B\ot \mu _A)=(\mu _A\ot id_B)\circ (id_A\ot R)\circ (id_A\ot \alpha _B^{-1}\ot id_A) 
\circ (R\ot id_A),
\label{ashom1}\\
&&R\circ (\mu _B\ot id_A)=(id_A\ot \mu _B)\circ (R\ot id_B)\circ (id_B\ot \alpha _A^{-1}\ot id_B)
\circ (id_B\ot R).
\label{ashom2}
\end{eqnarray}
\end{definition}
\begin{proposition}
If $R$ is an $(\alpha _A, \alpha _B)$-twisting map as above, then the linear map 
\begin{eqnarray*}
&&T:(A\ot B)\ot (A\ot B)\rightarrow (A\ot B)\ot (A\ot B), \\
&&T((a\ot b)\ot (a'\ot b'))=(\alpha _A(a)\ot b_R)\ot (a'_R\ot \alpha _B(b')),
\end{eqnarray*} 
is an $\alpha _A\ot \alpha _B$-pseudotwistor for the associative algebra $A\ot B$, 
with companions 
\begin{eqnarray*}
&&\tilde{T}_1=T_{13}\circ (\alpha _A^{-1}\ot \alpha _B^{-1}\ot id_A\ot id_B\ot id_A\ot id_B),\\
&&\tilde{T}_2=T_{13}\circ (id_A\ot id_B\ot id_A\ot id_B\ot \alpha _A^{-1}\ot \alpha _B^{-1}).
\end{eqnarray*}
The Hom-associative algebra $(A\ot B)_{\alpha _A\ot \alpha _B}^T$, called the 
{\em $(\alpha _A, \alpha _B)$-twisted tensor product} of $A$ and $B$, is denoted by 
$A(\alpha _A)\ot _RB(\alpha _B)$; its multiplication is $(a\ot b)(a'\ot b')=\alpha _A(a)a'_R
\ot b_R\alpha _B(b')$. 
\end{proposition}
\begin{proof}
Note first that (\ref{ashom1}) and (\ref{ashom2}) may be written is Sweedler-type notation as: 
\begin{eqnarray}
&&(aa')_R\ot b_R=a_Ra'_r\ot \alpha _B^{-1}(b_R)_r, \label{sweedashom1} \\
&&a_R\ot (bb')_R=\alpha _A^{-1}(a_R)_r\ot b_rb'_R. \label{sweedashom2}
\end{eqnarray}
We need to check the conditions (\ref{multT})-(\ref{pstwhom3}) for $T$; (\ref{multT}) follows 
immediately from (\ref{ashom0}). \\[2mm]
\underline{Proof of  (\ref{pstwhom1})}: We compute:\\[2mm]
${\;\;\;\;}$$((id_A\ot id_B\ot \mu _{A\ot B})\circ \tilde{T}_1\circ (T\ot id_A\ot id_B))
(a\ot b\ot a'\ot b'\ot a''\ot b'')$
\begin{eqnarray*}
&=&((id_A\ot id_B\ot \mu _{A\ot B})\circ 
T_{13}\circ (\alpha _A^{-1}\ot \alpha _B^{-1}\ot id_A\ot id_B\\
&&\ot id_A\ot id_B))
(\alpha _A(a)\ot b_R\ot a'_R\ot \alpha _B(b')\ot a''\ot b'')\\
&=&((id_A\ot id_B\ot \mu _{A\ot B})\circ T_{13})(a\ot \alpha _B^{-1}(b_R)\ot a'_R\ot 
\alpha _B(b')\ot a''\ot b'')\\
&=&(id_A\ot id_B\ot \mu _{A\ot B})(\alpha _A(a)\ot \alpha _B^{-1}(b_R)_r\ot a'_R\ot 
\alpha _B(b')\ot a''_r\ot \alpha _B(b''))\\
&=&(\alpha _A(a)\ot \alpha _B^{-1}(b_R)_r\ot a'_Ra''_r\ot 
\alpha _B(b'b'')\\
&\overset{(\ref{sweedashom1})}{=}&\alpha _A(a)\ot b_R\ot 
(a'a'')_R\ot \alpha _B(b'b'')\\
&=&T(a\ot b\ot a'a''\ot b'b'')\\
&=&(T\circ (id_A\ot id_B\ot \mu _{A\ot B}))(a\ot b\ot a'\ot b'\ot a''\ot b''), \;\;\;q.e.d.
\end{eqnarray*}
\underline{Proof of  (\ref{pstwhom2})}: We compute:\\[2mm]
${\;\;\;\;}$$((\mu _{A\ot B}\ot id_A\ot id_B)\circ \tilde{T}_2\circ (id_A\ot id_B\ot T))
(a\ot b\ot a'\ot b'\ot a''\ot b'')$
\begin{eqnarray*}
&=&((\mu _{A\ot B}\ot id_A\ot id_B)\circ 
T_{13}\circ (id_A\ot id_B\ot id_A\ot id_B\\
&&\ot \alpha _A^{-1}\ot \alpha _B^{-1}))
(a\ot b\ot \alpha _A(a')\ot b'_R\ot a''_R\ot \alpha _B(b''))\\
&=&((\mu _{A\ot B}\ot id_A\ot id_B)\circ 
T_{13})(a\ot b\ot \alpha _A(a')\ot b'_R\ot \alpha _A^{-1}(a''_R)\ot b'')\\
&=&(\mu _{A\ot B}\ot id_A\ot id_B)(\alpha _A(a)\ot b_r\ot \alpha _A(a')\ot b'_R\ot 
\alpha _A^{-1}(a''_R)_r\ot \alpha _B(b''))\\
&=&\alpha _A(aa')\ot b_rb'_R\ot 
\alpha _A^{-1}(a''_R)_r\ot \alpha _B(b'')\\
&\overset{(\ref{sweedashom2})}{=}&\alpha _A(aa')\ot (bb')_R\ot 
a''_R\ot \alpha _B(b'')\\
&=&T(aa'\ot bb'\ot a''\ot b'')\\
&=&(T\circ (\mu _{A\ot B}\ot id_A\ot id_B))(a\ot b\ot a'\ot b'\ot a''\ot b''), \;\;\;q.e.d.
\end{eqnarray*}
\underline{Proof of  (\ref{pstwhom3})}: Because of how $\tilde{T}_1$ and $\tilde{T}_2$ are 
defined, it is enough to prove the following relation:
\begin{eqnarray*}
&&(\alpha _A^{-1}\ot \alpha _B^{-1}\ot id_A\ot id_B\ot id_A\ot id_B)\circ (T\ot id_A\ot id_B)
\circ (\alpha _A\ot \alpha _B\ot T)\\
&&=
(id_A\ot id_B\ot id_A\ot id_B\ot \alpha _A^{-1}\ot \alpha _B^{-1})\circ 
(id_A\ot id_B\ot T)\circ (T\ot \alpha _A\ot \alpha _B). 
\end{eqnarray*}
We compute:\\[2mm]
$((\alpha _A^{-1}\ot \alpha _B^{-1}\ot id_A\ot id_B\ot id_A\ot id_B)\circ (T\ot id_A\ot id_B)
\circ (\alpha _A\ot \alpha _B\ot T))(a\ot b\ot a'\ot b'\ot a''\ot b'')$
\begin{eqnarray*}
&=&((\alpha _A^{-1}\ot \alpha _B^{-1}\ot id_A\ot id_B\ot id_A\ot id_B)\circ (T\ot id_A\ot id_B))
(\alpha _A(a)\ot \alpha _B(b)\\
&&\ot \alpha _A(a')\ot b'_R\ot a''_R\ot \alpha _B(b''))\\
&=&(\alpha _A^{-1}\ot \alpha _B^{-1}\ot id_A\ot id_B\ot id_A\ot id_B)
(\alpha _A^2(a)\ot \alpha _B(b)_r\\
&&\ot \alpha _A(a')_r\ot \alpha _B(b'_R)\ot a''_R\ot \alpha _B(b''))\\
&\overset{(\ref{ashom0})}{=}&(\alpha _A^{-1}\ot \alpha _B^{-1}\ot id_A\ot id_B\ot id_A\ot id_B)
(\alpha _A^2(a)\ot \alpha _B(b_r)\\
&&\ot \alpha _A(a'_r)\ot \alpha _B(b'_R)\ot a''_R\ot \alpha _B(b''))\\
&=&\alpha _A(a)\ot b_r
\ot \alpha _A(a'_r)\ot \alpha _B(b'_R)\ot a''_R\ot \alpha _B(b''), 
\end{eqnarray*}
$((id_A\ot id_B\ot id_A\ot id_B\ot \alpha _A^{-1}\ot \alpha _B^{-1})\circ 
(id_A\ot id_B\ot T)\circ (T\ot \alpha _A\ot \alpha _B))(a\ot b\ot a'\ot b'\ot a''\ot b'')$
\begin{eqnarray*}
&=&((id_A\ot id_B\ot id_A\ot id_B\ot \alpha _A^{-1}\ot \alpha _B^{-1})\circ 
(id_A\ot id_B\ot T))(\alpha _A(a)\ot b_r\\
&&\ot a'_r\ot \alpha _B(b')\ot \alpha _A(a'')\ot \alpha _B(b''))\\
&=&(id_A\ot id_B\ot id_A\ot id_B\ot \alpha _A^{-1}\ot \alpha _B^{-1})(\alpha _A(a)\ot b_r\\
&&\ot \alpha _A(a'_r)\ot \alpha _B(b')_R\ot \alpha _A(a'')_R\ot \alpha _B^2(b''))\\
&\overset{(\ref{ashom0})}{=}&(id_A\ot id_B\ot id_A\ot id_B\ot \alpha _A^{-1}\ot 
\alpha _B^{-1})(\alpha _A(a)\ot b_r\\
&&\ot \alpha _A(a'_r)\ot \alpha _B(b'_R)\ot \alpha _A(a''_R)\ot \alpha _B^2(b''))\\
&=&\alpha _A(a)\ot b_r
\ot \alpha _A(a'_r)\ot \alpha _B(b'_R)\ot a''_R\ot \alpha _B(b''), 
\end{eqnarray*}
and the two terms are obviously equal. 
\end{proof}
\begin{example}
Let $(A, \mu _A)$ and $(B, \mu _B)$ 
be two associative algebras and $\alpha _A:A\rightarrow A$ and 
$\alpha _B:B\rightarrow B$ two bijective algebra endomorphisms. 
Define the map $R:B\ot A 
\rightarrow A\ot B$, $R(b\ot a)=\alpha _A(a)\ot \alpha _B(b)$. Then one can 
easily check that $R$ is an $(\alpha _A, \alpha _B)$-twisting map, and 
$A(\alpha _A)\ot _RB(\alpha _B)$ coincides with $(A\ot B)_{\alpha _A\ot \alpha _B}$ as Hom-associative 
algebras. More generally, assume that $P:B\ot A\rightarrow A\ot B$ is a twisting map such that 
$(\alpha _A\ot \alpha _B)\circ P=P\circ (\alpha _B\ot \alpha _A)$. Define the map 
$R:B\ot A\rightarrow A\ot B$, $R=(\alpha _A\ot \alpha _B)\circ P$. Then one can 
check that $R$ is an $(\alpha _A, \alpha _B)$-twisting map, and 
$A(\alpha _A)\ot _RB(\alpha _B)$ coincides with $(A\ot _PB)_{\alpha _A\ot \alpha _B}$ 
as Hom-associative 
algebras.
\end{example}
\begin{example}
This is a particular case of the previous example, but can also be checked directly. Let $(A, \mu _A)$ 
be an associative algebra, $\sigma :A\rightarrow A$ an involutive algebra automorphism of $A$ 
and $q\in k^*$. Take $B=C(k, q)=k[v]/(v^2=q)$, take $\alpha _A=\sigma $ and $\alpha _B=id$ and 
define $R:B\ot A\rightarrow A\ot B$ a linear map with $R(1\ot a)=\sigma (a)\ot 1$ and 
$R(v\ot a)=a\ot v$, for all $a\in A$. Then $R$ is an $(\alpha _A, \alpha _B)$-twisting map. By 
using the formula of $R$, one can easily see that the multiplication of the Hom-associative 
algebra $A(\alpha _A)\ot _RB(\alpha _B)$ is given by 
\begin{eqnarray*}
&&(a\otimes 1+b\otimes v)(c\otimes 1+d\otimes v)=(\sigma (ac)+q\sigma (b)d)\otimes 1
+(\sigma (ad)+\sigma (b)c)\otimes v,  
\end{eqnarray*}
for all $a, b, c, d\in A$. If we consider again the associative algebra $\overline{A}$ obtained 
from $A$ by the Clifford process and the algebra automorphism 
$\overline{\sigma }:\overline{A}\rightarrow \overline{A}$, $\overline{\sigma }(a\ot 1+b\ot v)=
\sigma (a)\ot 1+\sigma (b)\ot v$, then one can see that $A(\alpha _A)\ot _RB(\alpha _B)=
(\overline{A})_{\overline{\sigma }}$ as Hom-associative algebras. 
\end{example}
%%%%%%%%%%%%%%%%%%%%%%%%

\end{document}